\documentclass[a4paper,12pt]{amsart}
\usepackage{fullpage}

\usepackage[colorlinks=true,pagebackref,hyperindex,pdfencoding=auto, psdextra]{hyperref}
\usepackage[capitalise]{cleveref}
\usepackage{mathrsfs} 

\usepackage[notcite,notref]{}
\usepackage[all]{xy}
\usepackage{color}
\usepackage{amsfonts}
\usepackage{amscd}
\usepackage{amssymb,latexsym,amsmath,amscd, graphicx} 
\usepackage{bbm}
\usepackage{ulem}

\definecolor{darkgreen}{rgb}{0.0, 0.7, 0.0}
\definecolor{cyan}{cmyk}{1,0,0,0}
\definecolor{gb}{RGB}{0,100,180}
\definecolor{dg}{RGB}{0,100,0}

\newcommand{\bdg}{\begin{dg}}
\newcommand{\edg}{\end{dg}}

\newcommand{\jchange}[1]{{#1}}
\newcommand{\mc}[1]{{#1}}

\newcommand{\jc}[1]{{#1}}
\newcommand{\jcc}[1]{{#1}}
\newcommand{\joc}[1]{{#1}}
\newtheorem{theorem}{Theorem}[section]
\newtheorem*{theorem*}{Theorem}
\newtheorem{lemma}[theorem]{Lemma}
\newtheorem{proposition}[theorem]{Proposition}
\newtheorem{remark}[theorem]{Remark}
\newtheorem{corollary}[theorem]{Corollary}

\newtheorem{??}[theorem]{Question}
\newtheorem{definition}[theorem]{Definition}

\newtheorem{notation}[theorem]{Notation}

\font\tenmsb=msbm10
\font\sevenmsb=msbm7
\font\fivemsb=msbm5
\newfam\msbfam
\textfont\msbfam=\tenmsb
\scriptfont\msbfam=\sevenmsb
\scriptscriptfont\msbfam=\fivemsb
\def\Bbb#1{{\fam\msbfam #1}}
\font\teneufm=eufm10
\font\seveneufm=eufm7
\font\fiveeufm=eufm5
\newfam\eufmfam
\textfont\eufmfam=\teneufm
\scriptfont\eufmfam=\seveneufm
\scriptscriptfont\eufmfam=\fiveeufm

\newcommand{\im}{ \hbox{\rm Im} }

\newcommand\comp{{\Bbb C}}

\newcommand\zed{{\Bbb Z}}

\newcommand{\ov}[1]{\overline{#1}}

\newcommand{\ms}[1]{\mathscr{#1}}

\newcommand{\ha}{\cA}

\newcommand{\sn}{S_{\ven}}
\newcommand{\snode}{S_{\ven}^{\times}}

\newcommand{{\socle}}{\rm Socle}

\newcommand{\ve}[1]{\underline{#1}}
\newcommand{\mn}{\cM_{n}}
\newcommand{\nn}{\cN_{n}^{d}}
\newcommand{\hn}{h_{n}}
\newcommand{\an}{\ensuremath{\cA_n}}

\newcommand{\hf}{\phi}
\newcommand{\can}{K_C}
\newcommand{\ven}{\ve{n}}

\newcommand{\dn}{d_{n}}

\newcommand{\EE}{\mathbbm{E}}

\newcommand{\QQ}{\mathbbm{Q}}

\newcommand{\cC}{\mathcal{C}}

\newcommand{\rE}{\mathrm{E}}
\newcommand{\rV}{\mathrm{V}}

\newcommand{\tensor}{\otimes}

\newcommand{\map}[1]{\stackrel{#1}{\longrightarrow}}

\newcommand{\un}[1]{\ensuremath{\protect\underline{#1}}}

\def\GL{\textrm{GL}}

\def\gl{\mathfrak{gl}}

\DeclareMathOperator{\Coh}{Coh}
\DeclareMathOperator{\Pic}{Pic}

\DeclareMathOperator{\Bun}{Bun_G(C)}
\DeclareMathOperator{\Gr}{Gr}

\DeclareMathOperator{\Higgs}{Higgs}
\DeclareMathOperator{\In}{In}
\DeclareMathOperator{\Nspan}{Nspan}
\DeclareMathOperator{\Flat}{Flat}

\DeclareMathOperator{\cHom}{\mathcal{H}om}
\DeclareMathOperator{\cEnd}{\mathcal{E}nd}
\DeclareMathOperator{\Isom}{Isom}
\DeclareMathOperator{\Spec}{Spec}
\DeclareMathOperator{\Ext}{Ext}

\DeclareMathOperator{\Aut}{Aut}

\DeclareMathOperator{\End}{End}
\DeclareMathOperator{\Sym}{Sym}
\DeclareMathOperator{\Lie}{Lie}
\DeclareMathOperator{\ad}{ad}

\DeclareMathOperator{\rank}{rank}

\DeclareMathOperator{\codim}{codim}

\DeclareMathOperator{\Trace}{trace}

\DeclareMathOperator{\Stab}{Stab}

\DeclareMathOperator{\Ad}{Ad}
\DeclareMathOperator{\mult}{mult}
\DeclareMathOperator{\tp}{top}

\def\proj{\mathrm{proj}}

\DeclareMathOperator{\RHL}{RHL}

\DeclareMathOperator{\act}{act}
\DeclareMathOperator{\aff}{aff}
\DeclareMathOperator{\car}{\mathbf{car}}
\DeclareMathOperator{\reg}{reg}
\DeclareMathOperator{\red}{red}
\DeclareMathOperator{\nod}{\times}

\def\univ{\textrm{\tiny univ}}

\def\1halb{\frac{1}{2}}
\def\tto{\twoheadrightarrow}

\def\sxymat{\xymatrix@C=1.5ex@R=0.8ex}
\def\grp{$\xymatrix{ R\times_{X}R  \ar[r]^-{\mu} & R \ar@<1ex>[r]^-{s}\ar@<-1ex>[r]_-{t} & X}$}
\def\dar{\ar@<-0.5ex>[r]\ar@<0.5ex>[r]}
\def\tar{\ar[r]\ar@<1ex>[r]\ar@<-1ex>[r]}
\newcommand{\dmap}[2]{\ar@<-0.5ex>[r]_-{#2}\ar@<0.5ex>[r]^-{#1}}
\newcommand{\dotarrow}[2]{\xymatrix{{#1}\ar@{..>}[r]&{#2}}}
\def\cart{\ar@{}[dr]|{\square}}

\def\cA{\mathcal{A}}

\def\cC{\mathcal{C}}

\def\cE{\mathcal{E}}
\def\cF{\mathcal{F}}

\def\cH{\mathcal{H}}
\def\cI{\mathcal{I}}

\def\cL{\ensuremath{\mathcal{L}}}
\def\cM{\mathcal{M}}
\def\cN{\mathcal{N}}
\def\cO{\mathcal{O}}
\def\cP{\mathcal{P}}

\def\cU{\mathcal{U}}

\def\cg{\mathfrak{g}}

\def\bA{{\mathbbm A}}

\def\bC{{\mathbbm C}}

\def\bE{{\mathbbm E}}

\def\bG{{\mathbbm G}}

\def\bL{{\mathbbm L}}

\def\bP{{\mathbbm P}}
\def\bQ{{\mathbbm Q}}
\def\bR{{\mathbbm R}}

\def\bV{{\mathbbm V}}

\def\bZ{{\mathbbm Z}}

\DeclareMathOperator{\dmult}{dmult}
\DeclareMathOperator{\dshift}{dshift}
\DeclareMathOperator{\nodes}{Nodes}

\DeclareMathOperator{\IC}{IC}

\title{A support theorem for the Hitchin fibration:\\ the case of $\GL_n$ and $\can$}
\author{Mark Andrea A.\ de Cataldo, Jochen Heinloth and Luca Migliorini}

\begin{document}

\begin{abstract}
	We compute the supports of the perverse cohomology sheaves of the Hitchin fibration for $\GL_n$ over the locus of reduced spectral curves. In contrast to the case of meromorphic Higgs fields we find additional supports at the loci of reducible spectral curves. Their contribution to the global cohomology is governed by a finite twist of Hitchin fibrations for Levi subgroups. The corresponding summands give non-trivial contributions to the cohomology of the moduli spaces for every $n \geq\jc{2}$. A key ingredient is a restriction result for intersection cohomology sheaves that allows us to compare the fibration to the one defined over  versal deformations of spectral curves.
\end{abstract}
\maketitle

\tableofcontents

\section{Introduction}\label{sec:Intro}

In the study of the geometry of the Hitchin fibration a recurring problem has been to determine how much of this geometry is determined by the smooth part of the fibration. Ng\^ o's support theorem provides a tool to formulate and sometimes to prove a precise version of this question for general fibrations equipped with an action of a family of polarized abelian group schemes (see \cite{ngoaf}). In particular, for variants of the fibration parameterizing Higgs bundles with poles,  Chaudouard and Laumon in \cite{ch-la} proved that the only perverse cohomology sheaves appearing in the decomposition of the direct image of the constant sheaf are the intermediate extensions of the local systems on the smooth locus, that is the perverse cohomology sheaves are supported over the whole base. In particular, all of the cohomology is determined, in principle, by the monodromy of the cohomology of the smooth fibers. As is explained in the last section of \cite{ch-la},
unfortunately, this method does not apply to the original symplectic  (no poles) version of the Hitchin fibration.
Motivated by the $P=W$ conjecture \cite{dCHM}, one would like to understand the perverse filtration of the fibration better and for this it is important to determine whether this result extends to this case as well. Surprisingly, we do find new supports \jc{as well as } new cohomological contributions for any rank $\jc{n\geq 2}$. 

Before we explain the general strategy of our approach, let us  state our main result. In order to do this, let us briefly introduce the standard notation that we use, which is recalled in more detail in \cref{sec:NotationAndBackground}. We fix a smooth projective curve $C$ and denote by $\hn\colon  \mn^d \to \an$ the Hitchin fibration for $\GL_n$ and  an integer $d$ coprime to $n$, i.e., $\mn^d$ is the moduli space of semistable Higgs bundles of rank $n$ and degree $d$ on $C$. The base $\cA_n$ is an affine space parameterizing spectral curves $C_a\in T^*C$ that are of degree $n$ over $C$.
For any partition $\un{n}=(n_i)_{i=1,\dots,{r}}$ of $n$ there is the closed subvariety $S_{\un{n}}\subset \cA_n$,  closure of the subset $S_{\un{n}}^\times\subset \sn$ of reducible nodal curves having smooth irreducible components of degree $n_i$ over $C$ (see {\cref{spcv}}). Also we denote by $\an^{\red}\subset \an$ the open subset parameterizing reduced spectral curves. 

{The decomposition theorem implies that the complex $\bR{\hn}_{*} \QQ$ is a direct sum of its perverse cohomology sheaves  $^p\!\!{\mathscr H}^{r}(\bR{\hn}_{*}\QQ)$, which in turn are direct sums of irreducible perverse sheaves. These summands are thus supported on closed subvarieties of $\an$ and the subvarieties that occur in this way are called the supports of  $\bR{\hn}_* \QQ$.}
Using these notions our main results can be summarized as follows (note that according to our convention in \S\ref{sec:NotationAndBackground}, the local systems given by the $r$-th cohomology of 
the smooth fibers of $\hn$ contribute to $^p\!\!{\mathscr H}^{r}(\bR{\hn}_{*}\QQ)${, the $r$-th perverse cohomology sheaf}):

\begin{theorem*}[\Cref{prop:SuppOnlySn} and \Cref{thm:main1}]
Let $\hn:\mn^d \to \an$ be the Hitchin map.
{The supports of  $\bR{\hn}_{*}\QQ$ on $\an$  that meet the reduced locus $\an^{\red}$ are exactly the strata $S_{\un{n}}$.}

Moreover, for every partition $\ven$ of $n$, the stratum $\sn$ is a support for all of the sheaves
	\begin{equation*}\label{eq:range}
	^p\!\!{\mathscr H}^{{k}}(\bR  {\hn}_{*}\QQ) 
	\hbox{ with } \delta^{\aff}(\ven) \leq {k} \leq {2\dim \an -\delta^{\aff}(\ven)} 
	\end{equation*}
	where $\delta^{\aff}(\ven)=\sum_{i<j}n_in_j(2g-2)-{r}+1$ is the dimension of the affine parts of the Picard group of { the spectral curves defined by points of $\sn^\times$}. The corresponding perverse summands are the intermediate extensions of local systems on $\sn^\times$ whose stalks can be described explicitly in terms of the cohomology of the spectral curve and its dual graph.
\end{theorem*}
A more precise statement describing the local systems appearing in the above statement appears in \Cref{thm:main1}, and refined information on the monodromy is given in \Cref{cor:rank_loc_syst}. In particular, it turns out that the local systems corresponding to partitions with pairwise distinct $n_i$ and $k$ maximal have trivial monodromy and therefore these contribute to $H^*(\mn^d) =H^*(\bR  {\hn}_{*}\QQ)$. \jc{For $n=2$ we describe the contributions of the summands supported on $S_{(1,1)}$ explicitly (\cref{rem:rank2}).} 

The key idea that allows us to get a hold on the supports of the perverse cohomology sheaves $^p\!\!{\mathscr H}^{{k}}(\bR{\hn}_{*}\QQ)$ of the direct image,  is to compare the  Hitchin fibration with a ``larger" fibration. In our situation, the description of Higgs bundles as sheaves on spectral curves allows us to use the fibration of compactified {J}acobians for a versal deformation of singular spectral curves. As is proved in \cite{MSV}, for those families only the full base is a support. To use this, we then study how the intersection cohomology sheaves on the versal family decompose under restriction to the Hitchin base. We prove a simple restriction result (\Cref{prop:splitperverse}) which we can then use for an explicit computation, because the Cattani--Kaplan--Schmid complex gives a rather explicit combinatorial description of the contribution to the top cohomology sheaves in the case of nodal curves. Interestingly, the combinatorial description is related to the bond matroid of the dual graph of the spectral curve, which luckily had been studied for entirely different reasons before.

It may be interesting to note that there is a simple geometric explanation for the different behavior between the Hitchin fibration we consider and the version with poles treated in \cite{ch-la}: If $a \in \sn$ corresponds to a reducible nodal spectral curve $C_a$ with ${r}$ nonsingular irreducible components, let $V$ be the base of its versal deformation. Then, the codimension in $V$ of the stratum $V^\times$ where all the nodes persist, equals the number of nodes, whereas the codimension of 
$\sn$ in the Hitchin base is smaller, {it is equal to $\delta^{\aff}(\ven)$ which differs from the number of nodes by $r-1$ (\cref{lem:deltasn}).}  In other words, {locally around a point $a\in \sn^\times$ the family of spectral curves over the Hitchin base $\an$ defines a morphism to the base of the versal defomation $V$, but the image of this morphism is not transversal to the stratum $V^\times$}.  This does not happen in the Hitchin fibration with poles, and
it is precisely this lack of transversality which is responsible for the splitting of the restriction of the intersection cohomology sheaves into summands (\Cref{rem:transverseorsplit}).

An elementary example of this phenomenon may be seen in the deformation of a curve consisting of two rational components meeting transversally in two nodes. The versal deformation has dimension two, every curve in the family except the central one is irreducible, and it is easily seen that only the full base is a support. In particular, denoting by $R^1$ the local system of first cohomology on the smooth locus, there is a non vanishing cohomology sheaf ${\mathcal H}^1(\IC(R^1))$ at the origin, accounting for the extra component. If we restrict this map to a disc passing through the origin, the total space of the family remains nonsingular but the restriction of $\IC(R^1)$ splits into two summands, one of which, supported at the origin, is precisely 
${\mathcal H}^1(\IC(R^1))$.


The structure of the article is as follows. In \Cref{sec:NotationAndBackground}, we set up notation and conventions.
In  \Cref{sec:supports_delta}
{we recall the main result from \cite{MS}, that constrains the potential supports of our perverse cohomology sheaves in terms of higher discriminants. The symplectic structure of Hitchin fibrations allows us to describe these in terms of the action of Jacobians of spectral curves.} For this we use that  the differential of  the Hitchin morphism $\hn$ is dual with respect to the symplectic form on $\mn^d$ to the infinitesimal action  of the abelian group schemes acting on the fibres. As it is hard to find this property of the Hitchin {fibration} explicitly in the literature, we include an algebraic proof in the  more general setup of the Hitchin fibration for complex reductive groups in the appendix (\Cref{prop:dactdhdual}). 
See also \cite{drs}, where this duality is established for a large class of integrable systems, including  our Hitchin fibration for $\GL_n$.
In \Cref{sec:supports_partions}, we combine the Ng\^o support theorem and results on compactified Jacobians to identify the strata $\sn$ as the only potential supports in $\an^{\red}$. Next, in \Cref{sec:comparison_versal} we prove the restriction result for $\IC$-sheaves  mentioned above (\Cref{prop:splitperverse}) and {show that it applies to} the Hitchin fibration by computing the Kodaira--Spencer map for the universal family of spectral curves. {In \Cref{sec:sn_supports} we use the Cattani--Kaplan--Schmid complex for the versal family to translate the problem of determining the generic fibers of the summands supported on $\sn$ into a combinatorial problem that we can then solve (\cref{thm:main1}).} {In \Cref{sec:mon_loc_sys} we  describe the monodromy of these summands, in order to prove that the new summands contribute to the global cohomology of the Hitchin fibration for any $n\geq \jc{2}$.}

\subsection*{Acknowledgments:} We would like to thank Fabrizio Caselli, Emanuele Delucchi, L\^e Minh H\'a, {Tam\'as Hausel}, \jc{Mirko Mauri,}  Luca Moci and Bau Chao Ng\^o for numerous discussions, comments and helpful advice during the long time that passed since we started working on this project. {We thank the referees for many helpful suggestions.}
Mark Andrea de Cataldo, who is partially supported by NSF grants DMS-1600515 and 1901975, would like to thank the Max Planck Institute for Mathematics in Bonn and the Freiburg Research Institute for Advanced Studies for the perfect
working conditions; the research leading to these results has received funding from the
People Programme (Marie Curie Actions) of the European Union's Seventh Framework
Programme (FP7/2007-2013) under REA grant agreement n.\ [609305].
Jochen Heinloth was partially supported by RTG 45 of the DFG.

\section{Notation and setup}\label{sec:NotationAndBackground}
Throughout this article we work over the complex numbers $\comp$. \jchange{We will use the analytic topology in order to work with constructible sheaves of $\bQ$-vector spaces. Readers preferring the \'etale topology, could recover all of our results in that topology by using $\bQ_\ell$-coefficients.}
\subsection{Conventions on intersection cohomology}\label{subsec:ConventionIC}
We start by recalling the basic results on intersection cohomology that we  need. 
To reduce the appearance of shifts of complexes in our results we will employ the following numbering convention for intersection cohomology sheaves, which differs from the one used in \cite{bbd}: 
Let $X$ be an algebraic variety,  $Y \hookrightarrow X$ be a closed subvariety and $\cL$  a local system on a smooth open subset  $j: Y^\circ \hookrightarrow Y$. 
We denote by $\IC(Y, \cL)$ the intersection cohomology complex with the normalization such that
$\IC(Y, \cL)_{|Y}=\cL[-\codim Y]$. In particular a local system on an open subset of $X$ will be put in cohomological degree $0$.  With this convention the strong support condition reads
\begin{align*}
&\cH^{l}(\IC(Y, \cL))=0 & \hbox{ if }l<\codim Y,  \\ 
&\cH^{\codim Y}(\IC(Y, \cL))=j_* \cL, & \\
&\codim \mathrm{Supp}\, \cH^{l}(\IC(Y, \cL))> l &\hbox{ for } l> \codim Y,
\end{align*}
i.e., this is the usual t-structure, but shifted by $\dim X$. This will be useful for us, as we will study restrictions of perverse sheaves to closed subvarieties and we can then avoid to shift the constant sheaf.  

A semisimple perverse sheaf on a complex variety $X$ is a complex of the form $P=\bigoplus_\alpha \IC(Y_\alpha, L_\alpha)$,
where $Y_\alpha \subseteq X$ are irreducible closed subvarieties and $L_\alpha$ are semisimple local systems defined 
on dense open subsets of the $Y_\alpha$'s. The generic points of the $Y_\alpha$'s 
are called the {\em supports} of $P$. 

If  
$h\colon M \to X$ is a proper map between smooth varieties, the decomposition theorem of \cite{bbd} says that 
$$\bR h_*\QQ \simeq \bigoplus_{k\geq 0}	 \,\,
^p\!\!{\mathscr H}^k(\bR h_*\QQ)[-k]$$
{and in addition for all $k$ the $k$-th perverse cohomology sheaf $^p\!\!{\mathscr H}^k(\bR h_*\QQ)$ is a semisimple perverse sheaf}. The union of supports of the perverse sheaves $^p\!\!{\mathscr H}^k(\bR h_*\QQ)$
is the set of {\em supports} of the map $h$  (see \cite[\S 7]{NgoLEmme}).

We say that a semisimple complex $K=\bigoplus_k \,\,
^p\!\!{\mathscr H}^k(\bR h_*\QQ)[-k]$ has {\em no proper supports} if $X$ is the only support of $K$.
\subsection{The Hitchin fibration}\label{subsec:HitchinFibration}
We fix a nonsingular, connected, projective curve $C$ of genus $g\geq 2$, an integer $n \in \zed_{\geq 1}$, and 
an integer $d \in \zed$  such that ${\rm gcd}(n,d)=1$. We  denote by $K_C$ the canonical bundle of $C$. 

We denote by $\Higgs_n^d$ the moduli stack of Higgs bundle of rank $n$ and degree $d$ on $C$, i.e., it parametrizes pairs $(E,\hf) $ where $E$ is a vector bundle of rank $n$ and degree $d$ on $C$ and $\hf\in H^0(C,\End(E)\tensor K_C)$. 

We denote by $\cM_{n}^{d}$ the coarse moduli space of stable Higgs bundles of rank $n$ and degree  $d$, where  as usual,  stability
is defined by imposing the inequality $\deg (F)/{\rm rank}(F) < \deg (E)/{\rm rank} (E)$ for every $\hf$-invariant
proper sub-bundle $F \subseteq E$. 

Because of our assumption that $n$ and $d$ are coprime $\cM_{n}^{d}$ is  (see \cite[Theorem 6.1 and Proposition 7.4]{Nitsure};
\mc{for the irreducibility, see \cite[\S2.1]{deca}}) an irreducible, nonsingular, quasi-projective variety of dimension  
\begin{equation}\label{eq:dimMn}
\dim (\cM_{n}^{d})= n^2(2g-2)+2=: 2\dn.
\end{equation}
The cotangent space $T^*\nn$ of the moduli space $\nn$ of stable rank $n$ and degree $d$ vector bundles on $C$ is a dense open subvariety of $\cM_{n}^{d}$.
The Hitchin base is defined to be the vector space
\begin{equation}\label{anz}
\an:= \prod_{i=1}^n H^0(C, \can^{\otimes i}),
\end{equation}
which has dimension
\begin{equation}\label{dan}
 \dim (\an)= \frac{1}{2} \dim (\cM_n^d) =   n^2(g-1)+1= \dn.
\end{equation}
The Hitchin morphisms
\begin{equation}\label{himo}
\un{h}_n^d\colon \Higgs_n^d \to \an \text{ and }
\hn^d: \cM_n^d \to \an
\end{equation}
assigns to any Higgs bundle $(E,\hf)$, the coefficients of the characteristic polynomial of $\hf$.
The morphism $\hn^d$ is proper, flat of relative dimension $\dn=n^2(g-1)+1$ {(e.g. \cite[Theorem 6.1]{Nitsure},\cite[Theorem II.5]{Faltings})}, by Stein factorization and the description of the generic fiber of $\hn$ recalled below (Theorem \ref{thm:BNR}) it has connected fibers, and it is often called the Hitchin fibration.

Since the degree $d$ doesn't play any role in what follows, as long as it is coprime to the rank $n$, we will not indicate it from now on, and simply write $\mn$ for  $\cM_n^d$
and $\hn$ for $\hn^d$.

\subsection{Spectral curves and the BNR-correspondence}\label{spcv}

As the key to the geometry of the fibers of the Hitchin fibration $h_n$ is their description as compactified Jacobians of spectral curves through the Beauville--Narasimhan--Ramanan--correspondence {(\cite{BNR})} we also recall this briefly.

Any {closed } point $a\in \an$ defines a curve $C_a$, called spectral curve, in the total space of the cotangent bundle $T^*C={\rm Tot}_C(\can)$ by viewing $a$ as a monic  polynomial of degree $n$ with coefficient of the degree $n-i$ term in $H^0(C,K_C^{\otimes i})$. This defines a flat family $C_{\an}\to \an$ of projective curves. {We will denote the fiber over any, not necessarily closed point $a\in\an$ by $C_a$. Everything we recall below for spectral curves over closed points holds for these as well, after base change to the residue field $k(a)$ of $a$.}

The natural projection $\pi: C_a \to C$, exhibits the spectral curve as a degree $n$ cover of $C$, but $C_a$ can be singular, non-reduced and reducible. {As by construction $\pi_* \cO_{C_a} \cong \oplus_{r=0}^{n-1}K_C^{\otimes -r}$ the family $C_\cA$ is a family of curves of arithmetic genus $n^2(g-1)+1= \dn$.}

We denote by $\an^{\mathrm{red}} \subset \an$ the subset corresponding to reduced spectral curves, by $\an^{\mathrm{int}} \subset\an^{\mathrm{red}} $ the subset corresponding to integral spectral curves and by $\an^{{\times}}\subset \an$ the open subset corresponding to spectral curves whose singularities are at worst nodes.
For us reducible spectral curves will be of particular interest.

When viewed as an effective  divisor on the surface $T^*C$, any spectral curve $C_a$
can be written uniquely as 
\begin{equation}\label{eq:spectralcurvedecomp}
C_a = \sum_{k=1}^{{r}} m_k{C}_{a_k}, 
\end{equation}
where the $a_k$ are the distinct irreducible factors of the characteristic polynomial and $m_k$ their multiplicities.
In particular, the $C_{a_k}$ are integral and pairwise distinct curves which are
spectral curves of some degree $n_k$. We  then have  
\begin{equation}\label{eq:DecompositionOfN}
n=\sum_{k=1}^r m_k n_k.
\end{equation}
For $\un{n}=(n_k)_k \in \bZ_{>0}^r$ we write
$$\cA_{\un{n}} := \prod_{k=1}^r \cA_{n_k}$$
Then, for $\un{n},\un{m}  \in \bZ_{>0}^r$  satisfying (\ref{eq:DecompositionOfN}), multiplication of polynomials $(p_k)_k \mapsto \prod p_k^{m_k}$ defines a finite morphism 
$$ \mult_{\un{m},\un{n}}\colon \cA_{\un{n}} \to \cA_n$$
and we denote by $S_{\un{m},\un{n}}$ its image. For { our results the  case } $\un{m}=\un{1}=(1,\dots,1)$ {is the most important one, as in this case } the generic point of the image consists of reduced spectral curves. {We will thus} abbreviate $$S_{\un{n}}:= S_{\un{1},\un{n}} \text{ and } \mult_{\un{n}}:=\mult_{\un{1},\un{n}}.$$
The generic spectral curves defined by points in these subsets are rather simple. 

\begin{lemma}\label{rem:bertini}
\begin{enumerate}
	\item[]
	\item For every $a\in \cA$ the spectral curve $C_a$ is connected.
	\item For every $\un{n},\un{m}$ satisfying $\sum m_kn_k=n$ there is a dense open  subset  $$S_{\un{m},\un{n}}^{\times} \subset S_{\un{m},\un{n}}$$ such that for $a\in  S^{\times}_{\un{m},\un{n}}$ the reduced curve $C_a^{\red} \subset C_a$ is nodal and with nonsingular irreducible components. 
	
	In particular, since every irreducible component of $C_a$ has genus $g\geq 2$, the curves $C_a^{\red}$ are stable curves in the sense of Deligne-Mumford \cite{DM} for all $a\in  S^{\times}_{\un{m},\un{n}}$.
\end{enumerate}
\end{lemma}
\begin{proof}
	This is a consequence of Bertini's theorem: {The spectral curves $C_{a_k} \subset T^*C \subset \bP:=\bP_C(\cO_C \oplus K_C)$  are defined by general sections of the relative $\cO_{\bP}(n_k)$ which has global sections $H^0(C,\Sym^{n_k}(\cO_C\oplus K_C))$. In particular for $n_k>1$ it induces a morphism to projective space that embeds $T^*C$ and thus the generic hyperplane section $C_{a_k}$ is smooth and connected. For $n_k=1$ all spectral curves $C_{a_k}$ are smooth and connected, as $C$ is. 
	Again by positivity of $K_C^{n_k}$ the open subset of $\cA_{\un{n}}$ where the curves intersect transversally is non-empty.}	
\end{proof}
%

\jchange{We will also need to recall the correspondence between Higgs bundles and torsion free sheaves on spectral curves, that was proved in increasing generality by Hitchin, Beauville--Narasimhan--Ramanan, and Schaub.}  
Given a Higgs bundle $(E,\phi)$ with $h_n(E,\phi)=a$ we can consider $E$ as a coherent sheaf on $C_a$, because sheaves on $T^*C$ can be viewed as $\cO_C$-modules equipped with an action of the $\cO_C$-algebra $\oplus_{i\geq 0} K_C^{-\otimes i}$.  The Cayley--Hamilton theorem then says that the module $\cF_{E,\hf}$ defined by $\phi$ is supported on $C_a$ and it is a torsion free sheaf of rank $1$ on $C_a${, i.e., the restriction map from local sections of $\cF_{E,\phi}$ to the preimage in $C_a$ of the generic point of $ C$ is injective and at all generic points of $C_a$ the sheaf has the same length as the structure sheaf $\cO_{C_a}$} (a notion that in the case of non-reduced curves was introduced by Schaub \cite[{D\'efinition 1.1 and Définition 1.2}]{Schaub}). Conversely given $\cF$ a torsion free sheaf of rank $1$ on $C_a$ the sheaf $E=\pi_*\cF$ is a vector bundle\jchange{, because it is a torsion free $\cO_C$-module on the smooth curve $C$}, it is of rank $n$\jchange{, because this is the length at the generic point of $C$ and it }comes equipped with a Higgs field $\hf$ \jchange{ induced from the $\cO_{C_a}$-module structure}.

\jchange{These constructions are inverse to each other and work without change for flat families of sheaves.} 
 \jchange{Since $\pi\colon C_a \to C$ is a finite morphism, $H^*(C_a,\cF)\cong H^*(C,\pi_{*}\cF)$ and thus the Euler characteristic of $\cF$ and the induced Higgs bundle $E$ agree.  Denoting by $\Coh_{1,C_\cA}^{tf} \to \cA$ the stack of torsion free sheaves of rank $1$ on spectral curves, we can summarize this as follows:}
\begin{theorem}[{\cite{Hitchin}\cite{BNR}\cite[{Proposition 2.1}]{Schaub}}]\label{thm:BNR}
	The functor $(E,\hf) \mapsto \cF_{E,\hf}$ induces an equivalence  $\Higgs_n\cong \Coh_{1,C_\cA}^{tf}$. Under this equivalence the stack $\Higgs_n^d$ is identified with the substack of torsion free sheaves of rank $1$ and Euler characteristic $\chi=d+n(1-g)$.
\end{theorem}
In \cite[{Th\'eor\`eme 3.1}]{Schaub} (see \cite[Remarque 4.2]{ch-la}) it was explained how stability of Higgs bundles translates into a stability condition for sheaves on spectral curves.

Let $a\in S_{\un{n}}$ be a point that defines a reducible, reduced spectral curve $C_a$ with irreducible components $C_{a_1},\dots,C_{a_{{r}}}$. 
In this case a torsion-free rank $1$ sheaf $\cF$ on $C_a$ defines a stable Higgs bundle of degree $d$ if and only if for all proper subcurves $C_I=\cup_{i\in I} C_{a_i} \subsetneq C$ we have
$$ \chi(\cF_{C_I}) \geq \sum_{i \in I} n_i \cdot (\frac{d}{n}+1-g),$$
where $\chi$ is the Euler characteristic and $\cF_{C_I}$ is the maximal torsion-free quotient of $\cF|_{C_I}$. 
\jchange{This coincides with the usual stability condition for Higgs bundles, as $\chi(\cF_{C_I})=\deg(\pi_*\cF_{C_I})+ \sum_{i\in I} n_i (1-g)$. In \cite[{Th\'eor\`eme 3.1}]{Schaub} this formula is written in terms of a normalized degree function $\deg_{X}(\cF)=\chi(\cF)-\chi(\cO_X)$, which simplifies to the above inequality as $\chi(\cO_{C_I})=(\sum_{i\in I} n_i)^2(1-g)$.}
\begin{remark}\label{rem:BNRstable}
	This notion of stability coincides with a stability notion for compactified Jacobians (see e.g., \cite[{Definition 2.11}]{MRV1}) with respect to the polarization $ \un{q}:=(n_i\cdot (\frac{d}{n}+1-g))_i$, which is a general polarization as $\gcd(n,d)=1$.
	
	In particular the restriction of the Hitchin fibration $\hn \colon \mn\to \an$ to $\cA^{\red}$ is a fine relative compactified Jacobian for the family $C_{\cA}|_{\cA^{\red}}$ in the sense of Esteves \cite[Theorem A]{EstevesTAMS}.
\end{remark}
Finally we recall the $\delta^{\aff}$-invariant of our spectral curves. For any spectral curve $C_a$ we denote by $J_a:=\Pic^{\un{0}}_{C_a}$ the generalized Jacobian of $C_a$, which is the group scheme parameterizing line bundles on $C_a$ that have degree $0$ on all irreducible components of $C_a$. The $J_a$ are the fibers of a group scheme $J_{\cA} \to \cA_n$ over $\cA$ which acts on $\mn$.

For every $a$ the connected group scheme $J_a$ has a canonical filtration 
$$ 0 \to J_a^{\aff} \to J_a \to J_a^{\proj} \to 0$$
where $J_a^{\aff}$ is affine, $J_a^{\proj}$ is projective and both are connected   {(e.g. \cite[Section 9.2, Corollary 11]{BLR})}. One defines
$$ \delta^{\aff}(C_a):= \dim( J_a^{\aff}).$$

\begin{remark}\label{rem:deltanode}{\cite[Section 9.2, Example 8]{BLR}}
	If $C_a$ is a reduced, connected curve and $\nu\colon \widetilde{C}_a \to C_a$ is the normalization, then $\nu^*$ defines an isomorphism $J_a^{\proj}\cong \Pic_{\widetilde{C}_a}^0$ {and $J_a^{\aff}=\ker(\nu^*)$}. In this case 
\begin{equation}\label{eq:deltaff}
\delta^{\aff}(C_a) =\dim  H^0(C,\nu_*\cO_{\widetilde{C}_a}/\cO_C)+1-\#(\pi_0(\widetilde{C}_a)).
\end{equation}
If furthermore, the only singularities of $C_a$ are nodes, we have
\begin{equation}\label{deltaff_nodes}
\delta^{\aff}(C_a) =\# (\text{nodes})+1-\#(\pi_0(\widetilde{C}_a))=1-\chi(\Gamma)=\dim H^1(\Gamma),
\end{equation}	
where $\Gamma$ is the dual graph of the curve $C_a$.
\end{remark}

The function $a \mapsto \delta^{\aff}(C_a)$ is upper semi-continuous by \cite[X, Remark. 8.7]{sga3.2.2}, i.e. there are closed subsets
$$\cA_n^{\geq \delta}:=\{ a \in \cA_n | \delta^{\aff}(C_a)\geq \delta\} \subseteq \cA_n.$$

\begin{notation}
For a flat family $C_Y \to Y$ of projective curves over an irreducible scheme $Y$ with generic point $\eta_Y$, we will denote by $\delta^{\aff}(Y):= \delta^{\aff}(C_{\eta_Y})$ and call it generic $\delta^{\aff}$-invariant on $Y$.	
\end{notation}	

\begin{lemma}\label{lem:deltasn}
	Let $\un{n}$ be a partition of $n$ and ${\eta_{\un{n}}}\in S_{\un{n}}$ the generic point then we have 
	$$\codim S_{\un{n}}= \delta^{\aff}(C_{{\eta_{\un{n}}}})=:\delta^{\aff}(\un{n})$$
	{ and $$\delta^{\aff}(\ven) = \sum_{i<j} n_in_j(2g-2)-r+1.$$} 
\end{lemma}
\begin{proof}
	We know that $\dim S_{\un{n}} = \dim \cA_{\un{n}} = \sum_{i=1}^{{r}} (n_i^2 (g-1)+1).$
	
	By \cref{rem:bertini} for a general $a=(a_{{i}})\in \cA_{\un{n}}$ the spectral curve $C_a$ {is a connected curve which }has ${r}$ smooth components intersecting transversally. As each component is defined by a polynomial in $\cA_{n_i}=\oplus_{{k}=1}^{n_i} H^0(C, K_C^{\tensor {k}})$ we have 
$$ \# (C_{a_i}\cap C_{a_j}) = n_in_j (2g-2).$$
By Remark \ref{rem:deltanode} {this implies $\delta^{\aff}(\ven) = \sum_{i<j} n_in_j(2g-2)-r+1$ and thus we find} 
 \begin{align*}
 \dim S_{\un{n}} + \delta^{\aff}(C_a) &=  \left(\sum_{i=1}^{{r}} n_i^2 (g-1)\right) + {r}  + \sum_{i<j} 2(n_in_j)(g-1) - {r} +1
 \\ & = n^2(g-1)+1 = \dim \an{.}
 \end{align*}
\end{proof}	
{
\begin{notation}\label{not:Gamma-n} For a partition $\ven$ of $n$ we will denote by $\Gamma_{\ven}$ the dual graph of any spectral curve given by a point of $S_{\ven}^\times$, i.e., it is the graph with vertices $\{1,\dots,r\}$ corresponding to the irreducible components of the curve and $n_in_j(2g-2)$ edges between the vertices $i,j$, corresponding to the intersection points of the components. 
\end{notation}
}

\section{The supports have to be $\delta$-loci}\label{sec:supports_delta}
In this section, we show that supports of the complex $\bR h_*\bQ$ can only be irreducible components of the subschemes $\cA_n^{\geq \delta}\subset \cA_n$ of spectral curves of $\delta^{\aff}$ invariant at least $\delta$. 
A result of this type appears in \cite{ch-la}. 
Here, we  give a different argument, relating the computation of the higher discriminants of  the Hitchin fibration to the $\delta^{\aff}$ invariant.


Let us recall the notion of higher discriminants of a map from \cite[{Section 1.3}]{MS}:
\begin{definition}\label{def:high_discr}
	Let $f: X \to Y$ be a proper map between complex nonsingular varieties. For any $i\geq 1$, the $i$-th discriminant 
	$\Delta^i (f)$ is  the  locus of $y \in Y$ such that
	there is no $(i-1)$ dimensional subspace of $T_y Y$ transverse to $df_x(T_x X)$ for every $x \in f^{-1}(y)$, i.e.\
	such that the preimage of an $(i-1)$-dimensional disc around $y$ would be nonsingular of codimension $\dim Y - i+1$.
\end{definition}
The $i$-th discriminants $\Delta^i (f)$ form a decreasing sequence of closed subsets and moreover $\Delta^1(f)$ is the discriminant locus of the map $f$, i.e.\ the complement of the biggest open subset of $Y$ where the restriction of the morphism $f$ is a smooth morphism.
{As explained in \cite[Section 8]{MiglioriniSupportTheorems}, the existence of Whitney stratifications \cite[Theorem 4.14]{VerdierStratifications} implies that there is a stratification of $\an$ by smooth locally closed subvarieties such that the preimage under $h$  of transversal slices to the strata are smooth and therefore}
\begin{equation}\label{codest}
\codim \Delta^k (f)\ge k \mbox{ for all } k.
\end{equation}

The relevance of higher discriminants stems from the following:
\begin{theorem} \cite[{Theorem B}]{MS}\label{hd_general}
	Let $f\colon X \to Y$ be a projective map between smooth varieties.   
	{Then any support of $\bR f_* \QQ$ that has codimension $k$ in $Y$ has to be an irreducible component of $\Delta^k(f)$.}
\end{theorem}

For the Hitchin fibration, the higher discriminants turn out to be $\delta$-loci.

\begin{theorem}\label{thm:hidiscrHitc}
	Let $\hn:\mn \to \an$ be the Hitchin map.
	The $i$-th discriminant of $\hn$ is equal to 
	\begin{equation}\label{hidiscrHitc}
	\Delta^i(\hn)=\cA_n^{\geq i}=\{a \in \an \text{ such that } \delta^{\aff}(C_a) \geq i \}.
	\end{equation}
\end{theorem}
\begin{proof} To prove that $\Delta^i(\hn)\supseteq\cA_n^{\geq i}$
	let us denote by $\act \colon J_{\cA_n} \times_{\ha_n} \mn \to \mn$ the action given by the tensor product of line bundles with torsion free sheaves.
	By Proposition \ref{prop:dactdhdual} the differential $d\act$ is dual to the differential $d\hn$ of $\hn$ with respect to the symplectic form $\omega_{\Higgs}$ on $\mn$, {i.e., for any point $m\in \mn$ with $\hn(m)=a$ the differential $d\act_m \colon \Lie(J_a) \to T_m\mn$ of the action $\act_m\colon J_a \to \mn$ given by $\act_m(j)=j.m$ is dual to $dh_{n,m}\colon T_m\mn \to T_a\ha_n$.}
	
	For any $a\in {\cA_n}$ we defined $\delta^{\aff}(C_a)$ to be the dimension of the affine part $J_a^{\aff}$ of $J_a=\Pic^{\un 0}(C_a)$ and this group scheme acts on the projective fiber ${\hn}^{-1}(a)\subset \mn$. By Borel's fixed point theorem there exists a fixed point $m=(E,\hf) \in {\hn}^{-1}(a)$ for the action of { the commutative affine group scheme} $J_a^{\aff}$.
	Therefore $\Lie(J_a^{\aff})$ is in the kernel of {$d\act_m$} and by the duality this implies that $a\in \Delta^{\delta^{\aff}(C_a)}$. 
	
	Conversely we know that for any $m\in {h_n}^{-1}(a)$ the stabilizer $\Stab_{J_a}(m)$ is affine (e.g., because  {for a} rank $1$ torsion free sheaf $\cF$ on $C_a$ and {any line bundle} $\cL\in J_a$ the pull backs $\nu^*\cF$ and $\nu^*(\cF\tensor \cL)$.  to the normalization of $C_a^{\red}$ can only be isomorphic if $\nu^*\cL$ is {$m$-torsion for some $m\leq n$}). Thus we know that at any point $m$ the kernel of $d \act$ is contained in $\Lie(J_a^{{\aff}})$ and therefore { by duality we have} $\Delta^{\delta}(\hn)\subseteq {\cA_n}^{\geq \delta}$.	
\end{proof}
Since  we  have already noted that $\codim \Delta^k (f)\ge k \mbox{ for all } k$ (\cref{codest}), the following is an  immediate consequence of \cref{thm:hidiscrHitc}:
\begin{corollary}\label{cor:codest_delta}
Let $\hn:\mn \to \an$ be the Hitchin map then the codimension of the $\delta$-loci of the Hitchin base satisfy
\begin{equation}
\codim \cA_n^{\geq i} \geq i.
\end{equation}
In particular, since the relative Picard group of  the spectral curve family {$\cC_{\cA}$}  is polarizable (see {\cite[Proposition 4.12.1]{NgoLEmme}}\cite[{{Theorem 3.3.1} }]{deca}), 
{the Hitchin map} is a $\delta$-regular weak abelian fibration in the sense of  Ng\^o \cite[{Section 7.1}]{ngoaf}.
\end{corollary}
	
\begin{remark}\label{rem:Sndeltacomponent} By Lemma \ref{lem:deltasn}, the subvarieties $S_{\un{n}}$ have codimension equal to their generic $\delta^{{\aff}}$ invariant $\delta^{\mathrm{aff}}(\un{n})$.
{By Theorem \ref{thm:hidiscrHitc} i}t follows that they are $\delta^{\mathrm{aff}}(\un{n})$-codimensional components of $\Delta^{\delta^{\mathrm{aff}}(\un{n})}(\hn)$. In view of Theorem
\ref{hd_general}, this makes these subvarieties  potential supports of $\bR h_{n,*}\QQ$.
\end{remark}

\section{The supports have to be partition strata}\label{sec:supports_partions}

Recall that $\cA^{\red}$ is the open subset of $\cA$ parameterizing reduced spectral curves. The main result of this section is the following proposition:
 
\begin{proposition}\label{prop:SuppOnlySn}
	If {a subvariety} $Y\subset \cA$ {with $Y\cap \cA^{\red}\neq \varnothing$ }is a support of $\bR h_*\QQ$, then $Y=S_{\un{n}}$ for some partition $\un{n}$ of $n$.
\end{proposition}
The proof of this result will be a simple combination of Ng\^o's support theorem with information on irreducible components of compactified Jacobians. Let us recall these results. Ng\^o proved a general result on the supports for the cohomology of a projective morphism $h\colon X \to Y$ which is a 
$\delta$-regular weak abelian fibration. {By \cref{cor:codest_delta} this condition is satisfied for the Hitchin fibration $\hn\colon \mn \to \an$ and in this case the result reads as follows. \mc{Recall that due to the nonsingularity of $\mn$, the constant sheaf $\QQ$ on it is self-dual, i.e. it coincides with  its own  Verdier dual, up to shift}.}

{\begin{theorem}[{\cite[Theorem 7.2.1]{NgoLEmme}}]\label{thm:NgoSupp}
	If $Y\subset \an$ is a support of $\bR h_*\QQ$ then the highest cohomology sheaf $\bR^{top} h_* \QQ$ contains a summand supported at $Y$.
\end{theorem}}

For us the main aspect of this theorem is that supports can only appear where the set of irreducible components of the fibers is not locally constant\jcc{, more precisely for any stratification of $\an^{\red}$ such that the restriction of $\bR^{top} h_* \QQ$ to every stratum is locally constant, the supports of $\bR h_*\QQ$ meeting $\an^{\red}$ have to be among the closures of the strata.}

{By decomposing reduced spectral curves into their irreducible components we know that $\cA_n^{\red}$ is the union of the images of the maps $\mult_{\un{n}}\colon (\prod_{k=1}^r \cA_{n_k}^{{\mathrm{int}}})^{\red} \to \cA_n^{\red}$ where $\ven$ runs through the partitions of $n$. As the generic points of these images are the generic points of the strata $S_{\ven}$, we can prove Proposition \ref{prop:SuppOnlySn} simply by showing that the restriction of  the highest cohomology sheaf $\bR^{top} h_* \QQ$ is locally constant on these images. By proper base change this follows from the following Lemma.
}  

\begin{lemma}\label{lem:SuppOnlySn}
	For any partition $\un{n}$ of $n$ let $\mult_{\un{n}}\colon (\prod_{k=1}^r \cA_{n_k}^{{\mathrm{int}}})^{\red} \to \cA_n^{\red}$
	be the restriction to $\cA_n^{\red}$ of the morphism given by the multiplication of characteristic polynomials. Then the sheaf $\mult_{\un{n}}^*\bR^{top} h_* \QQ$ is constant.
\end{lemma}
\begin{proof}
As $h$ is flat and proper  {the formation of $\bR h_*\QQ =\bR h_!\QQ$ commutes with base change and therefore the } stalks of  $\mult_{\un{n}}^*\bR^{top} h_* \QQ=  \mult_{\un{n}}^*\bR^{top} h_! \QQ$ have a basis indexed by the irreducible components of the fibers.
{As explained in {\cref{rem:BNRstable} the BNR-correspondence shows that}}  $h^{-1}(\cA^{\red}) \to \cA^{\red}$ is a fine relative compactified Jacobian for the family of spectral curves which have planar singularities.  This property is stable under any base change by definition. By \cite[Corollary 2.20]{MRV1} the subvariety parameterizing locally free sheaves on the spectral curves is dense in every fiber { and therefore the sheaf of irreducible components of the compactified Jacobian is a subsheaf of the sheaf of components of the Jacobian of the spectral curve.} Moreover the irreducible components of the Jacobian of a spectral curve are indexed by the degrees of the restrictions of the line bundles to the irreducible components of the underlying curve {\cite[Section 9.3, Corollary 14]{BLR}} and {by definition	} stability is determined by a numerical condition on these degrees. As the sheaf of irreducible components of $C_a$ is constant on $(\prod_{k=1}^r \cA_{n_k}^{\mathrm{int}})^{\red}$
we deduce that the sheaf of irreducible components of the Jacobians of $C_a$ is also constant on this space and thus the same is true for the components of the fibers of $h$. This shows our claim.
\end{proof}

\section{Comparison with versal families}\label{sec:comparison_versal}
To check that the strata $\sn \subset \cA_n^{red}$ corresponding to reducible and generically reduced spectral curves are supports we will compare the family of spectral curves at a general point of $\sn$ to a versal deformation.
For the versal family of nodal curves we know from \cite[{{Theorem 1.8}}]{MSV} 
that the cohomology of the corresponding family of compactified Jacobians has full support and that the Cattani--Kaplan--Schmid complex allows to compute the fibers of the corresponding intersection cohomology complexes at every point. 

In order to deduce the decomposition also for the Hitchin fibration, we need to control the behavior of the intersection cohomology complexes under restriction and compute the Kodaira--Spencer map for the family of spectral curves. These results are proven in this section.

\subsection{Splitting the restriction of an intersection cohomology complex}
We begin with an easy fact concerning restriction of a semisimple intersection cohomology complex under the hypothesis that the restriction remains semisimple and satisfies a Hard Lefschetz type symmetry. This happens for example if the complex arises from a projective morphism and the restriction is taken to a closed subvariety with a smooth preimage.
  
The argument is reminiscent of the characterization of supports as ``relevant strata'' in a semismall map in \cite{BM,dCM}. {As in these references we will denote by $\mathrm{D}^b_c(X)$ the category of bounded complexes with constructible cohomology sheaves on a variety $X$.} \jchange{We start with an easy observation.}
\begin{lemma}\label{lem:restrIC}
	Let $\cL$ be a semisimple local system on an open dense subset $U\subset X$ and let  $P=\IC(\cL)$ be its intersection cohomology complex. 
	Let ${\iota}:Z \hookrightarrow X$ be a closed subvariety such that: 
	\begin{enumerate}
		\item 
		$U \bigcap Z$ is Zariski dense in $Z$.
		\item
		The complex ${\iota}^*P$ is perverse semisimple.
	\end{enumerate}
	Then 
	$$ {\iota}^*P = \oplus_k \IC(Z^{k},\cL^{k})$$
	where $Z^k$ is the union of the irreducible components of $\mathrm{Supp}\, \cH^{k}(P)\bigcap Z$  of codimension $k$ in $Z$ and $\cL^k = \cH^k({\iota}^*P)$ on the smooth part of the dense open subset of $Z^k$ where this sheaf is a local system.	
\end{lemma}
\begin{proof}
	Recall that, if $Q$ is a perverse semisimple sheaf on $Z$, 
	then we have a canonical decomposition 
	\begin{equation}\label{candec}
	Q=\bigoplus_{k=0}^n \IC(Z^k, \cL^k)
	\end{equation}
	where, for every $k$,
	$Z^k$ is a  closed subvariety of $Z$ of codimension $k$, and  $\cL^k$ is a semisimple local systems on an open set  $Z^{k,\circ}$ of ${Z^k}$.
	Note that $Z^k$ is allowed to be reducible and $\cL^k$ may have different rank on the different components of $Z^{k,\circ}$.
	
	The subsets $Z^k$ and local systems $\cL^k$ afford an easy characterization, which follows immediately from the strong support condition
	for the intersection cohomology complex (\Cref{subsec:ConventionIC}), i.e.: 
	For every $k$, the closed subset $Z^k$ is  the union of the $k$-codimensional components of $\mathrm{Supp}\, \cH^{k}(Q)$ and if $x \in Z^{k,\circ}$, then there is a canonical isomorphism $\cL^k_{x}= \cH^{k}(Q)_{x}.$ \joc{Taking $Q=\iota^*P$ this proves our claim.}
\end{proof}

\begin{remark}\label{rem:transverseorsplit}
	Notice that, by the support condition (see \cref{subsec:ConventionIC}) for the intersection cohomology complex, we have that
	$\codim \mathrm{Supp}\, \cH^{k}(\IC(\cL)) \geq k+1$. Thus, 
	{for a sufficiently generic closed subvariety $Z$} we have
	$\mathrm{Supp}\, \cH^{k}(\IC(\cL))\bigcap Z $ has codimension at least $k +1$ in $Z$ and therefore it cannot contribute a perverse summand. On the other hand,
	if $\codim \mathrm{Supp}\, \cH^{k}(\IC(\cL))\bigcap Z < k$, then  ${\iota}_*{\iota}^*\jchange{\IC(\cL)}$ is not perverse.
	 So if ${\iota}^*\jchange{\IC(\cL)}$ is perverse on $Z$ then the \jchange{ summands with smaller support occur  because $Z$ is not transversal to the supports of the cohomology sheaves.}
\end{remark}

\jchange{In our application the condition of the above lemma will follow from a relative hard Lefschetz theorem. We formalize this as follows.} 
\begin{definition}
	\label{classS}
	We write $\RHL(X) \subseteq \mathrm{D}^b_c(X)$ for the collection of  semisimple complexes such that there exists an integer $m$ and a decomposition
	\begin{equation}
	K= \bigoplus_{i=0}^{2m} \,\, ^p\!\!{\mathscr H}^{i}(K)[-i]
	\end{equation}
	that is \emph{RHL symmetric} for $m$ in the sense that 
	\begin{equation}
	^p\!\!{\mathscr H}^{m+i}(K)\simeq \, ^p\!\!{\mathscr H}^{m-i}(K)(-i) \hbox{ for every }i=0,\dots, m
	\end{equation}
	where we denoted the $i$-th Tate twist by $(i)$.
\end{definition}

Using this definition, we can state the main result of this section, which allows to describe the restriction of a complex $K$ in $\RHL(X)$ to a closed subscheme $Z$ in terms of the cohomology sheaves of $K$ under the condition that the restriction happens to lie in $\RHL(Z)$.

\begin{proposition}\label{prop:splitperverse}
	Let $X$ be an algebraic variety of dimension $n$, $K \in \RHL(X)$ a complex with no proper supports and $U\subset X$ a dense open subset {small enough to assure that}  $K\simeq \bigoplus \IC(\cL_i)[-i]$ for some local systems $\cL_i$ on $U$.
	
	Assume $\iota:Z \hookrightarrow X$ is a closed locally complete intersection subvariety of codimension $c$ in $X$ such that
	$Z\bigcap U$ is dense in $Z$ and with the property that
	\begin{equation}\label{goodcond}
	\iota^*K\in \RHL(Z).
	\end{equation}
	
	For every $i, k$
	set 
	\begin{equation}
	\widetilde{Z_{i}^k}=  \mathrm{Supp}\, \cH^{k}(\IC(\cL_i)) \bigcap Z
	\end{equation}
	and let $Z_{i}^k$ be the {union of the $k$-codimensional components of $\widetilde{Z_{i}^k}$} in $Z$.
	
	Then we have
	\begin{equation}\label{candec1}
		^p\!\!{\mathscr H}^k({\iota}^*K) \simeq {\iota}^* \, ^p\!\!{\mathscr H}^k(K) = \bigoplus_k \IC(Z_{i}^k, \cL_{i}^k)
	\end{equation}
	where
	\begin{equation}
	\cL_{i}^{k}=\cH^{k}(\IC(\cL_i))
	\end{equation}
	on the dense open set of $Z_{i}^k$ where $\cH^{k}(\IC(\cL_i))$ is a {non-zero} local system.
\end{proposition}
\jchange{The statement \cref{candec1} implies that the restrictions $\iota^*\IC(\cL_i)$ are again perverse sheaves}. In particular all irreducible components of $\widetilde{Z_i}^k$ have codimension $\geq k$. In the special case $k=\dim Z$ the result can therefore be rephrased as follows.
\begin{corollary}\label{cor:point}
	In the hypotheses above, let $p \in Z$ be a closed point. 
	Then $p$ is the support of a summand in the decomposition \cref{candec1} if and only 
	if $\cH^{\dim Z}(\IC(\cL_i))_p \neq 0$ {for some $i$}. 
\end{corollary}

\begin{proof}[\jchange{Proof of \Cref{prop:splitperverse}:}] \jchange{By \Cref{lem:restrIC} we only need to show that the restriction ${\iota}^* \, ^p\!\!{\mathscr H}^k(K)$ is still a perverse sheaf.} 	Applying $c$ times \cite[Corollaire 4.1.10]{bbd},   we have that, for every $i$,
	$$
	{\iota}^* \, ^p\!\!{\mathscr H}^k(K) \in \mathrm{D}^b_c(Y)^{[0,c]}.
	$$
	Suppose $k_0$ is the biggest integer for which 
	${\iota}^* \, ^p\!\!{\mathscr H}^{k_0}(K)$ is not perverse.
	By the symmetry assumption we can assume $k_0 \geq m$. Moreover, as by our semisimplicity assumption 
	$${\iota}^* \, ^p\!\!{\mathscr H}^{k_0}(K)=\oplus_{j=0}^d P_j[-j]$$
	for some perverse sheaves $P_j$\jchange{, a} non-zero summand $P_{j}$ with $j>0$ would contribute a summand in 
	$\, ^p\!\!{\mathscr H}^{k_0 +j}({\iota}^*K)$ which violates the RHL symmetry.
\end{proof}
\begin{remark}\label{purity} In the situation of Proposition \ref{prop:splitperverse}
	assume $K$ is pure of weight $0$ so that by our assumptions $\jcc{\iota}^*K$ is pure of weight $0$ too. 
	Then the local systems $\cL_{i}^{k} $ are pure of weight $i+k$.
	Notice that since ${\cL}_i$ is of weight $i$, by purity we have 
	$$
	\mathrm{weight} (\cH^k(\IC({\cL}_i)))_x \leq i+k,
	$$
	therefore the local systems  $\cL_{i}^{k}$ are the maximal weight quotients of the cohomology sheaves.
\end{remark}

\begin{remark}\label{forex}
	The assumptions of  \cref{prop:splitperverse} are met when  \jchange{$K=\bR f_*\QQ$ for a projective map $f\colon M\to X$ from a smooth variety $M$ such that $\bR f_*\bQ$ has no proper supports and  $Z\subset X$ is a local complete intersection such that $f^{-1}(Z)$ is nonsingular}\joc{, as in this situation the decomposition theorem and the relative hard Lefschetz theorem apply to both $f$ and its restriction $ f_{|f^{-1}(Z)}: f^{-1}(Z) \longrightarrow Z$}. 
\end{remark}

\subsection{The Kodaira--Spencer map for spectral curves}

We want to apply Remark \ref{forex} to compare the cohomology of the Hitchin fibration to the cohomology of relative compactified Jacobians for versal families of spectral curves. To verify the assumptions that $Z$ is a local complete intersection we need to describe the Kodaira--Spencer map for the family of spectral curves over $\ha_n$. {(An introduction to Kodaira-Spencer maps can be found in \cite[Section 3]{TalpoVistoli}, for the general deformation theory see \cite[Chapitre III]{Illusie}.)} 

For any point $a\in \cA_n$ we denote by $\cI_{C_a} \subset \cO_{T^*C}$ the ideal sheaf defining $C_a \subset T^*C$. 
Recall that embedded deformations of $C_a \subset T^*C$ are described by the cotangent complex 
$$ \bL_{C_a/T^*C} =[\cI_{{C_a}}/\cI_{{C_a}}^2 \to 0]$$   
which {is} concentrated in degree $-1$. Considering the composition $C_a \hookrightarrow T^*C \to \Spec k$ we see that the cotangent complex of $C_a$ is 
$$ \bL_{C_a} = \left[\cI_{{C_a}}/\cI_{{C_a}}^2 \to \left(\Omega_{T^*C}|_{C_a}\right)\right].$$
Now the universal spectral curve over $\ha_n$ defines a Kodaira--Spencer map
\[ KS_a\colon T_a\ha_n \to H^1(C_a,\bL_{C_a}^\vee) = \Ext^1(\bL_{C_a},\cO_{C_a}). \]
We know that the $\bG_m$-action on $\cA_n$ and the translation action $H^0(C,\can)\times \ha_n \to \ha_n$ lift to the universal spectral curve $C_{\ha_n}\to \ha_n$ and therefore induce trivial deformations of $C_a$. 
Let us denote by 
\[ \dmult\colon {H^0(C,\cO_C)} \to T_a\cA_n\cong \oplus_{i=1}^n H^0(C,K_C^{\otimes i}) \]
the derivative of the $\bG_m$-action and by
\[ \dshift \colon H^0(C,K_C) \to  T_a\cA_n\cong \oplus_{i=1}^n H^0(C,K_C^{\otimes i}) \]
the derivative of the translation. We will show in \Cref{Lem:KodairaSpencerComputation} below that the span of the image of these maps is the kernel of the Kodaira--Spencer map.

Let us also recall that $S_{n,1}\subset \cA_n$ is the locus of spectral curves that are given by the $n$-th infinitesimal neighborhood of a section in $T^*C$.
\begin{lemma}[Kodaira--Spencer map for $C_a$]\label{Lem:KodairaSpencerComputation}
	\begin{enumerate}
		\item[]
		\item For any point $a\in \ha_n-S_{n,1}$ the kernel of the Kodaira--Spencer { map $KS_a$} is the direct sum of the images of $\dmult$ and $\dshift$, i.e., the map $KS_a$ factors as 
		$$ T_a \cA_n \tto (T_a \cA_n)/(H^0(C,\cO_C\oplus \can)) \hookrightarrow H^1(C_a,\bL_{C_a}^\vee).$$
		\item 	For $a \in S_{n,1} \subset \ha_n$ the kernel of the Kodaira--Spencer { map $KS_a$} is equal the image of $\dshift$, i.e., the map $KS_a$ factors as 
		$$ T_a \cA_n \tto (T_a \cA_n)/(H^0(C,K_C)) \hookrightarrow H^1(C_a,\bL_{C_a}^\vee).$$
	\end{enumerate}
\end{lemma}
\begin{proof}
	Let us first describe the sheaves occurring in $\bL_{C_a}$ more explicitly.
	The cotangent bundle $T^*C$ is the relative spectrum of the $\cO_C$ algebra 
	$$\Sym^\bullet K_C^{\otimes -1} = \oplus_{r=0}^\infty \can^{\otimes -r}$$ and the spectral curve $	C_a\subset T^*C$ is defined by the ideal generated by { the image of the morphism } $K_C^{\otimes -n} \to  \oplus_{r=0}^\infty K_C^{\otimes -r}$ defined by $\alpha \mapsto \alpha + a_1\alpha + \dots + a_n\alpha$. Therefore, denoting by $\pi_a \colon C_a \to C$ the projection we see that
	\[\pi_{a,*} \cO_{C_a} \cong \oplus_{r=0}^{n-1}K_C^{\otimes -r}\]
	and
	\[ \cI_{{C_a}}|_{C_a}= \cI_{C_a}/\cI_{C_a}^2 \cong \pi_a^* K_C^{\otimes -n}.\]

	The dual of the canonical map 
	$$  \bL_{C_a}=  \left[\cI_{{C_a}}/\cI_{{C_a}}^2 \to \left(\Omega_{T^*C}|_{C_a}\right)\right] \to [\cI_{{C_a}}/\cI_{{C_a}}^2 \to 0]= \bL_{C_a/T^*C} $$
	is given by
	$$ [0 \to 	(\cI_{{C_a}}/\cI_{{C_a}}^2)^\vee ] \to [ T_{T^*C}|_{C_a} \to (\cI_{{C_a}}/\cI_{{C_a}}^2)^\vee ].$$
	Note that
	$$  H^0(C_a,(\cI_{{C_a}}/\cI_{{C_a}}^2)^\vee) \cong H^0(C, \oplus_{r=1}^n K_C^{\tensor{r}})=T_a\an$$
	is the space of embedded deformations of $C_a\subset T^*C$.
	
	Taking cohomology of the exact triangle of complexes:
	$$  \to  [0 \to 	(\cI_{{C_a}}/\cI_{{C_a}}^2)^\vee ] \to \bL_{C_a}^\vee  \map{p}  [ T_{T^*C}|_{C_a} \to 0]  \to  $$
	we obtain a long exact sequence:
	\begin{equation}\label{seq:kslong} 0 \to H^0(C_a,\bL_{C_a}^\vee) \map{H^0(p)}  H^0(C_a, T_{T^*C}|_{C_a}  ) \map{\delta} H^0(C_a, (\cI_{{C_a}}/\cI_{{C_a}}^2)^\vee) \map{KS_a} H^1(C_a,\bL_{C_a}^\vee) \to \dots\end{equation}
	To conclude we will compute the dimension of $ H^0(C_a, T_{T^*C}|_{C_a})$ and then compare it to the dimension of the image of $\dmult$ and $\dshift$.
	
	Restricting the relative tangent sequence $0\to \pi^* K_C \to T_{T^*C} \to \pi^* TC \to 0$  on $T^*C$ to $C_a$ we get 
	$$ 0\to \pi_a^* K_C \to T_{T^*C}|_{C_a} \to \pi_a^* K_C^{\otimes -1} \to 0.$$
	Applying $\pi_{a,*}$ and the projection { formula }we find: 
	$$ 0 \to \oplus_{r=0}^{n-1} \can^{\otimes (1-r)} \to \pi_{a,*} (T_{T^*C}|_{C_a}) \to \oplus_{r=0}^{n-1} \can^{\otimes (-1-r)} \to 0.$$
	In particular we see that \[H^0(C_a, T_{T^*C}|_{C_a}  ) \cong H^0(C_a,\pi_a^*K_C)=H^0(C,\can) \oplus H^0(C,\cO_C).\]
	Thus we have an exact sequence:
	\begin{equation}\label{eq:KSmap}
 H^0(C_a,\pi_a^*K_C)=H^0(C,\can) \oplus H^0(C,\cO_C) \map{\delta} T_a\an \map{KS_a} H^1(C_a,\bL_{C_a}^\vee),
	\end{equation}
	where the map $\delta$ is determined by the differential in the cotangent complex $\bL_{C_a}^\vee$.
	 
	Now let us determine the dimension of the image of $\dmult$ and $\dshift$. The $\bG_m$ action is given by the action of weight $i$ on $H^0(C,K_C^{\tensor i})$, so at $a$ the element $c \in \bC=Lie(\bG_m)$ defines the tangent vector $(a_i+ i a_i\cdot\epsilon ) \in T_a\an \subset \an(\bC[\epsilon]/(\epsilon^2))$. 
	
	Similarly as $a$ is given by the coefficients of a characteristic polynomial the translation by an element $\omega \in H^0(C,K_C)$ sends a polynomial $p(t)$ to $p(t-\omega)$. Thus the derivative at $a=(a_i)$ is $(a_i - (n-i+1) \omega  a_{i-1} \cdot\epsilon)$ where we put $a_0:=1$.
	
	In particular these vector fields are linearly independent unless $a_{i}= (-1)^i {n \choose i} \omega^i$, i.e., $a\in S_{n,1}$. This shows that the the kernel of $KS_a$ has dimension $\geq g+1$ for $a\not\in S_{n,1}$. By equation (\ref{eq:KSmap}) we know that the dimension is $\leq g+1$, so this shows the first claim.
	
	If $a$ is the $n-$fold multiple of a section, then the spectral curve $C_a$ admits a continuous family of automorphisms, given by multiplication of the nilpotent generator therefore $ H^0(C_a,\bL_{C_a}^\vee)$ which is the tangent space to the automorphism group of $C_a$ is at least $1$ dimensional. {From the above computation of the image of $\dshift$ we  know that the image of $\delta$ has to be at least $g$ dimensional. Combining these two observations with the exact sequence (\ref{seq:kslong}), we see that both estimates have to be equalities. This implies the second claim.}
\end{proof}

Let us now apply this result to the restriction of the Hitchin fibration to the subset of nodal curves.  We denote by $\overline{\ms{M}}_{d_n}$ the stack of stable curves of genus $d_n=\dim \cA_n$. Then by \cref{rem:bertini} the {flat} universal family of spectral curves $C_{\an}$ induces a morphism 
\[f_{\nod}\colon \an^{\nod} \to  \overline{\ms{M}}_{d_n}.\]
{Recall from \cref{rem:BNRstable} that for any $a\in \cA^{\nod}$ the stability condition for Higgs bundles corresponds to the stability condition for rank $1$ torsion free sheaves on the curve $C_a$ defined by the general polarization $\un{q}=(n_i\cdot(\frac{d}{n}+1-g))_i$. Here $n_i$ \jchange{is the degree of the irreducible component $C_i$ of $C_a$ over $C$. The arithmetic genus of a spectral curve of degree $n_i$ is  $g(C_i)=n_i^2(g-1)+1$ (\joc{see} \ref{dan}) and therefore  $n_i=\sqrt{\frac{g(C_i)-1}{g-1}}$ can be expressed in terms of the genus of $C_i$. Similarly a union of components $C_I=\cup_{i\in I} C_i$ defines a spectral curve of degree $n_I:=\sum_{i\in I} n_i $ over $C$ and therefore we again have $n_I= \sqrt{\frac{g(C_I)-1}{g-1}}$.}
		
This allows us to define a \joc{ compatible family of }polarization\joc{s} on \jchange{the Zariski open neighborhood $\cU\subset \overline{\ms{M}}_{g_n}$ of $f_{\nod}(\an^{\nod})$ parameterizing nodal curves with irreducible components of arithmetic genus equal to $g(C_I)$ for some $I\subset \{1,\dots r\}$, as follows. For $u\in \cU$ corresponding to a curve $C_u$ with irreducible components $(C_j)_{j=1,\dots l}$ of genus $g_j=g(C_{I_j})$ define $q_j:= \sqrt{\frac{g_j-1}{g-1}}\cdot (\frac{d}{n}+1-g)$. As we have seen above, 	the terms  $\sqrt{\frac{g(C_I)-1}{g-1}}$ are integers and $q_j$ is a generic polarization in the sense of \cite{MRV1} \joc{and by construction these polarizations are compatible under deformations of subcurves as in \cite[Definition 5.3]{MRV1}}.}

The following is a consequence of the work of Esteves \cite[Theorem A]{EstevesTAMS} and Melo--Rapagnetta--Viviani \cite[Theorem C]{MRV1}.
\begin{proposition}
	\jchange{Over the open neighborhood } ${\cU} \subset \overline{\ms{M}}_{g_n}$ \jchange{ of $f_{\nod}(\an^{\nod})$ defined above there exists} a regular and irreducible Deligne-Mumford stack ${\pi}\colon \overline{J}_{{\cU}}(\un{q}) \to {\cU}$ that étale locally is a relative compactified Jacobian parametrizing $\un{q}$-stable rank 1 torsion free sheaves.
\end{proposition}
\begin{proof} \jchange{We first observe that it suffices to know that for any point $a\in {\cU}$ there exists an \'etale neighborhood $U \to \cU$ such that a regular and irreducible compactified Jacobian exists over $U$. As these spaces are geometric coarse moduli spaces of the algebraic stack of $\un{q}$-stable rank $1$ torsion free sheaves they are canonically isomorphic on the intersections of these neighborhoods and therefore define an \'etale covering of {a} stack $\overline{J}_{{\cU}}(\un{q})\to {\cU}$. }
		
 \jchange{Now by \cite[Theorem A]{EstevesTAMS}, compactified Jacobians exist  for any family of proper reduced curves, when stability is taken with respect to a polarization induced from a vector bundle over the family of curves; \jc{moreover,}  any polarization is locally of this form (\cite[Remark 2.16]{MRV1}), because \'etale locally one can construct vector bundles having specified degrees on the irreducible components \joc{(e.g., one can use direct sums of line bundles defined by suitable local sections through smooth points of the irreducible components)}.  }

 \jchange{The regularity of $\overline{J}_{{\cU}}(\un{q})$ can be checked locally. By \cite[Theorem C]{MRV1} for any general polarization $\un{q}$ the relative compactified Jacobians $\overline{J}_{\Spec R}(\un{q})$ are regular and irreducible whenever $R$ is the complete local ring given by an effective versal deformation of a reduced locally planar curve and this implies the regularity of  $\overline{J}_{{\cU}}(\un{q})$.}
\end{proof}	

Combining this result with our computation of the Kodaira--Spencer map for the family $C_{\an}$	 (\cref{Lem:KodairaSpencerComputation}) we deduce: 

\begin{corollary}\label{cor:map_to_moduli}
	For every {non-trivial} partition $\ven $ of $n$, let $a \in \sn^{\times}$ (see \cref{rem:bertini}).
	Given a subvariety $ \Sigma_a ${ of $\an^\times$} passing through $a$ and  intersecting $\sn^\times$ transversally, 
	the classifying map $f_{\Sigma_a}\colon\Sigma_a \to {\cU} \subset \overline{\ms{cM}}_{g_n}$ is unramified on an open neighborhood of $a$ in $\Sigma_a$.
	Furthermore we have a cartesian diagram 
	\begin{equation}\label{equ:cartesian}
	\xymatrix{
		{\hn}^{-1}(\Sigma_a) \ar[rd]_{{\hn}_{|\jc{\Sigma_a}}}\ar[r]^{\simeq} &  \overline{J}_{{\cU}}(\un{q})\times_{{\cU}} \Sigma_a \ar[r]\ar[d]  & \overline{J}_{{\cU}}(\un{q}) \ar[d]^{{\pi}}\\
		&\Sigma_a \ar[r]^-{f_{\Sigma_a}}  & {\cU}.
	}
	\end{equation}
\end{corollary}

\begin{proof}
{
The strata $\sn^{\times}$ are invariant under the scaling action of $\bG_m$ and the translation by elements in $H^0(C,K_C)$. Thus, the tangent space of $\Sigma_a$, which is assumed to intersects $\sn^{\times}$ transversally at $a$, will be transversal to the kernel of the Kodaira-Spencer map $KS_a$ at $a$ by \cref{Lem:KodairaSpencerComputation}, i.e., $f_{\Sigma_a}$ is unramified at $a$ and therefore the same holds in an open neighborhood of $a$.    		}
\end{proof}	

\jc{\begin{corollary}\label{cor:compwithversal}With the notation and assumptions of \cref{cor:map_to_moduli} we have for any $k\in \bZ$:
		$${}^p\cH^k(\bR h_{n|\Sigma_a,*}\bQ) \cong f_{\Sigma_a}^*\big( {}^p\cH^k(\bR \pi_*\bQ)\big).$$
	\end{corollary}
\begin{proof}
	As {$\Sigma_a$ and $\cU$ are smooth and} the map $f_{\Sigma_a}$ is unramified \jchange{it is \'etale locally a closed embedding \cite[Tag 04HJ]{stacksproject} of a smooth variety and therefore it is \'etale locally a local complete intersection morphism.}  As perverse cohomology sheaves can be determined étale locally, the complexes $\bR h_{n|\Sigma_a,*}\bQ$ and $\bR \pi_*\bQ$ satisfy the RHL condition (\cref{forex}). \joc{Moreover, we noted in the beginning of this section that \cite[{{Theorem 1.8}}]{MSV} 
	says that the cohomology of compactified Jacobians over $\cU$, which is a versal family of nodal curves, has full support, so that} the restriction result for semisimple complexes \cref{prop:splitperverse}(\ref{candec1}) thus applies to $f_{\Sigma_a}\colon \Sigma_a \to \cU$.
\end{proof}}

\section{The partition strata are supports}\label{sec:sn_supports}

{Our next aim is to use the restriction result of the previous section to show that the strata $S_{\ven}$ are supports and to describe the local systems that give rise to the summands supported on $S_{\ven}$. The starting point is the main result of \cite{MSV} that shows that for a versal family of nodal curves the corresponding family of compactified Jacobians has full supports and moreover the corresponding $IC$ complexes have a rather explicit description in terms of the Cattani--Kaplan--Schmid (CKS) complex. Applying the restriction theorem of the previous section to this explicit description we reduce the computation to a combinatorial problem, that can be formulated in terms of matroids and drawing from results on matroids we can then conclude our main result.}

{Throughout this section we will consider the following setup. We will denote by $\pi\colon \cC \to B$ a flat projective versal family of locally planar curves and by $U \subset B$ the open subscheme over which the morphism $\pi$ is smooth. The local system over $U$ defined by the first cohomology group of the fibers of $\pi$ will be abbreviated as
\begin{equation*}
R^1:=(\bR^1\pi_*\QQ)_{|U}.
\end{equation*}
In this setting the main theorem of \cite{MSV} reads as follows.}
\begin{theorem}[{\cite[Theorem 5.1{2}]{MSV}}]\label{thm:hd_cj}
	Let $\pi:\cC \to B$ be a projective versal family of curves with locally planar singularities and arithmetic genus $g$, and let  $\pi^J:\ov{J}_{\cC}\to B$ 
	be a relative fine compactified Jacobian. 
	Then we have 
	\begin{equation*}
	\bR \pi^J_* \QQ= \bigoplus_{i=0}^{2g} \IC(\bigwedge^i R^1)[-i],
	\end{equation*}
	i.e., the complex $\bR \pi^J_* \QQ$ has no proper supports on $B$.	
\end{theorem}

\subsection{Description of the $\IC$ complexes for families with full support}
{In order to use this we now recall the explicit description of $\IC(\bigwedge^i R^1)$ in terms of the CKS complex}  introduced in \cite{cks, KK} {and described for deformations of nodal curves in \cite[Section 3]{MSV}.} 
	
{Let us first recall the general result.} Assume $B$ is a complex manifold of dimension $n$, $D \subset {B}$ is a normal crossing divisor $D$, and
$\cL$ is a local system on 
$B \smallsetminus D$  with  unipotent monodromies $\{T_i\}$ around the components of $D$. We work locally, near a point $p \in {D}$. After picking a holomorphic chart $U \subset B$ in a neighborhood of $p$, we may assume $p$ to be the origin in a polydisc $\Delta^n$ and the divisor $D$ to have equation $\prod_{i=1}^l z_i=0$. Thus $U \bigcap (B \smallsetminus D)\simeq (\Delta^*)^l \times \Delta^{n-l} $, where $\Delta^*$ is the punctured unit-disc.
Up to taking a slice transverse to the stratum of $D$ to which $p$ belongs, we may assume $l=n$, and denote $i_p: \{*\} \to B$ the closed embedding. The local system on  $(\Delta^*)^n$ is described by the stalk at a base point, a vector space $L$, and $n$ commuting nilpotent endomorphisms $N_i =\log T_i: L \to L$.
Given a subset $\{i_1, \cdots i_k\}=I \subset \{1, \cdots , n\}$, with $1 \leq i_1 <i_2<\cdots ,< i_k \leq n$, we set $N_I= N_{i_1}N_{i_2}\cdots N_{i_k}$, where we remark that the order of the composition doesn't matter as the endomorphisms commute. We denote  by $|I|$ the cardinality of $I$, and consider the complex  
\begin{equation}\label{cks_general}
{\mathbf C}^{\bullet}( \{ N_j \}, L):=\{
0 \to  L \to \bigoplus_{|I|=1} \im N_I  \to \bigoplus_{|I|=2} \im N_I \to \cdots \to  \im N_{\{1, \cdots , n\}}  \to 0 \},
\end{equation} 
where $L$ is in degree zero, and
where the differentials are given up to the standard signs by
\begin{equation}\label{diff_cks}
N_r:\im N_{i_1}\cdots N_{i_k}\to \im  N_r N_{i_1}\cdots N_{i_k} \hbox{ if } r \notin \{i_1, \cdots , i_k\}.
\end{equation}
If the local system underlies a variation of pure Hodge structures of weight $k$, 
${\mathbf C}^{\bullet}( \{ N_j  \}, L)$ is in a natural way a complex of mixed Hodge structures (\cite[\S 4]{KK}) 
isomorphic to $i_p^*\IC(\cL)$ \cite[Corollary 3.4.4]{KK}, and
its cohomology  sheaves 
$\cH^r(i_p^*\IC(\cL))$ have a natural mixed Hodge structure and \jchange{its  weight filtration} satisfies: 
\begin{equation}\label{e:purity}\cH^r(i_p^*\IC(\cL))=W_{r+k}\cH^r(i_p^*\IC(\cL)),
\end{equation}
(\cite[Theorem 4.0.1]{KK}).

\begin{remark}
The weight filtration used in \cite{KK} differs by a shift from  the one used in \cite{cks}.
The one in \cite{KK} gives the statement in the form of \cref{e:purity}, which is compatible with the MHS on the fibre via the decomposition theorem,
while the one in \cite{cks} gives  
$\cH^r(i_p^*\IC(\cL))=W_{k}\cH^r(i_p^*\IC(\cL))$, \cite[Corollary 1.13]{cks}.
\end{remark}

We describe the complex of mixed Hodge structures
${\mathbf C}^{\bullet}( \{ N_j  \}, L)$  in the case of a family of relative Jacobians associated with a family of stable curves.

We start from a stable nodal  curve $C_\times$, of arithmetic genus $g$ and dual graph $\Gamma$.
{We abbreviate $\delta^{\aff}:=\delta^{\aff}(C_\times)$}, which, since $C_\times$ is connected, equals $\dim H^1(\Gamma)$  by \cref{deltaff_nodes}.
Let $\widetilde{p}:\widetilde{C}_{\times} \to {C}_\times$
be the normalization map.
\jc{Let $\ms{M}_g \subseteq \ov{\ms{M}_g}$ be the moduli stack of semistable curves of genus $g$.}
Let ${B}$ be an {\'etale} neighborhood  of $[C_\times] \in \overline{\ms{M}_{g}}$ on which there exists a universal family $\pi\colon\cC \to {B}$, and let $D$ be the preimage of the boundary divisor $\overline{\ms{M}_{g}} \smallsetminus  \ms{M}_{g}$ in ${B}$. 

If ${B}$ is small enough, the irreducible components of $D$ are in natural one-to-one correspondence with the nodes of $C_\times$: given a node $e$, the general point of the corresponding irreducible component $D_e$ of $D$ corresponds to a curve of the family where the node $e$ persists while the other nodes are smoothed.
On $U:=B\smallsetminus D$ we have
the local system 
\begin{equation*}
R^1:=(R^1\pi_*\QQ)_{|U}.
\end{equation*}
 {T}he stalk {of $R^1$} at a base point $\eta \in U$ {is} $H^1(\cC_\eta,\bQ)$, a rational vector space of dimension 
 \jchange{$2g(\cC_\eta)$} endowed with a family  of unipotent commuting endomorphisms: for each node $e$ there is  $\{T_e\}_{e }$, the monodromy around the component $D_e$. 
The space $H^1(\cC_\eta,\bQ)$ is endowed with the weight filtration 
\begin{equation}\label{eq:weight_psi}
W_0=H^1(\Gamma,\bQ)\subseteq W_1=H^1(C_\times,\bQ) \subseteq W_2=H^1(\cC_\eta,\bQ),
\end{equation}
described as follows:

The first inclusion is induced by the short exact sequence of sheaves on $C_\times$ associated with the normalization map:
\begin{equation}\label{eq:NormalizationSequence}
0 \to {\bQ}_{C_\times} \to \widetilde{p}_* {\bQ}_{\widetilde{C}_\times} \to \oplus_{p \in \nodes(C_\times)} \bQ_{p} \to 0,
\end{equation}
 which induces the sequence: 
$$ 0 \to  H^1(\Gamma,\bQ) \to H^1(C_\times, \bQ) \to H^1(\widetilde{C}_\times,\bQ) \to 0,$$
giving the weight filtration of the Mixed Hodge structure $H^1(C_\times)$.
The second inclusion in \cref{eq:weight_psi} is induced by the exact sequence arising from the specialization \jchange{ sequence that computes $H^*(C_{\eta},\bQ)$ as cohomology of nearby cycles on $C_\times$ (see \cite[Section 3.0.2]{MSV})}:
$$ 0 \to  H^1(C_\times, \bQ) \to H^1({\cC}_\eta,\bQ) \to H_1(\Gamma,\bQ)(-1)\to  0.$$
The graded quotients associated with the filtration \cref{eq:weight_psi} are:
\begin{eqnarray*}
Gr^W_0 H^1 ( \mathcal{C}_{{\eta}}, \bQ) & = & H^1(\Gamma, \bQ), \\
Gr^W_1 H^1 ( \mathcal{C}_{{\eta}}, \bQ) & = & H^1(\widetilde{C}_\times, \bQ), \\
 Gr^W_2 H^1 ( \mathcal{C}_{{\eta}}, \bQ) & = & H_1 (\Gamma, \bQ)(-1).
\end{eqnarray*}
Here $H^1(\Gamma, \bQ)$ and $H_1(\Gamma, \bQ)$ are endowed with a pure Hodge-Tate structure of type $(0,0)$.

More precisely, if $\bE$ (resp. $\bV$) denote the vector space generated by the edges (resp. the vertices) of the dual graph, {choosing an orientation of the edges} we {obtain} the complexes
\begin{equation}\label{eq:graph_hom}
0 \to \bE \to \bV \to 0, \qquad 0 \to \bV^* \to \bE^* \to 0,
\end{equation}
computing respectively the homology and the cohomology of $\Gamma$, so that we identify $H_1(\Gamma, \bQ)$ with a subspace of $\bE$ and
$H^1(\Gamma, \bQ)$ with a quotient of $\bE^*$.
For every node $e$ there is an operator $N_e:=\jc{\log(T_e)}\colon H^1(\cC_\eta,\bQ) \to H^1(\cC_\eta,\bQ)\jc{(-1)}   $, which, 
by the Picard--Lefschetz formula, factors as
\begin{equation}\label{eq:PL_formula}
H^1(\cC_\eta,\bQ) \tto H_1(\Gamma,\bQ)\jchange{(-1)} {\map{N_e^\prime}} H^1(\Gamma, \bQ)\jc{(-1)}\subset H^1(\cC_\eta,\bQ)\jc{(-1)},
\end{equation}
and is given by
\begin{equation}\label{E:mapNe}
N_e^{{\prime}}\colon H_1(\Gamma, \bQ) \jchange{(-1)} \hookrightarrow \EE \xrightarrow{t \mapsto \langle {e}^*, t \rangle \cdot {e}^*} \EE^*  \twoheadrightarrow H^1(\Gamma, \bQ)\jc{(-1)}
\end{equation}
where $e$ is an orientated edge $e^*$ is its dual element in  $\EE^*$ (note that the  formula above for $N_e$ is independent of the {choice of} orientation of $e$).
\begin{remark}\label{rem:limitMHS}
With every one-dimensional family of nonsingular curves degenerating to $C_\times$ is associated
a limit mixed Hodge structure (\cite{Schmid, Steenbrink}) on the rational cohomology of a general fibre, more canonically on the cohomology of the nearby fibre. It follows from the above Picard--Lefschetz formula (\cref{E:mapNe} and \cref{eq:PL_formula}) that the filtration on $H^1(\cC_\eta,\bQ)$  defined above (\cref{eq:weight_psi}) coincides with the weight filtration of the limit mixed Hodge structure with respect to any smoothing family. The factorization \cref{eq:PL_formula} corresponds to the fact that the logarithm of monodromy is an endomorphism of type $(-1,-1)$ on the limit mixed Hodge structure \cite[Theorem 6.16]{Schmid}. 
\end{remark}

The direct image local systems for the relative compactified Jacobian family over $U$ are the exterior powers $\bigwedge^iR^1$ for $i=0,\cdots,2\jchange{g(\cC_\eta)}$. For every subset $I$ of edges we have the operators $N_I$ on $\bigwedge^i H^1(\cC_\eta, \bQ)${, induced by the operators $N_e$}.
The restriction of the intersection complex of $\bigwedge^iR^1$  to the point of $B$ {that corresponds to the nodal curve $C_\times$}  is computed by the complex 
${\mathbf C}^{\bullet}( \{ N_j \}, \bigwedge^i H^1(\cC_\eta\jc{,\bQ}))$ (\cref{cks_general}).  
The weight filtration on $\bigwedge^i H^1(\cC_\eta, \bQ)$ induced by the one on $H^1(C_\eta,\bQ)$ has a highest weight quotient given by:
\begin{align}\label{eq:weight_exterior}
&Gr_{2i}^W(\bigwedge^i H^1(\cC_\eta, \bQ))= (\bigwedge^i H_1(\Gamma, \bQ))(-i) && \hbox{ if }i < \delta^{\aff},  \\ 
&Gr_{i+\delta^{\aff}}^W(\bigwedge^i H^1(\cC_\eta, \bQ)))= (\bigwedge^{\delta^{\aff}} H_1(\Gamma, \bQ))(-\delta^{\aff})\otimes \bigwedge^{i-\delta^{\aff}}H^1(\widetilde{C}_{\times},\bQ) &&\hbox{ for } \jcc{\delta^{\aff} \leq i \leq 2g(C_\eta)-\delta^{\aff}}.
\end{align}

We recall a key observation from \cite{MSV}, \jchange{which is a simple linear algebra computation using the explicit formula for the $N_e$:} 
\begin{lemma}[{{\cite[Lemma 3.6]{MSV}}}]\label{lem:bond_matroid}
The vector space $\im \, N_I \subseteq \bigwedge^i H^1(\cC_\eta, \bQ)\jc{(-|I|)}$ is non-zero only if $|I|\leq i$ and $\Gamma \smallsetminus I$ is connected. 
In particular, $\im \, N_I =0$ if $|I|> \delta^{\aff}= \dim H^1(\Gamma,\bQ)$ for every $i$.

{Moreover, if $\Gamma\smallsetminus I$ is connected, the highest weight \jchange{quotient} $\Gr^W_{\tp}(\im \, N_I)$  of $\im \, N_I $ is isomorphic to 
$$\begin{array}{ll}
\Gr^W_{2i}(\im \, N_I) \cong \left(\bigwedge^{i -|I|} H_1(\Gamma\smallsetminus I,\bQ)\right)(-i) & \text{if } i\leq \delta^{\aff} \\
\Gr^W_{i+\delta^{\aff}}(\im \, N_I)\cong \left(\bigwedge^{\delta^{\aff}-|I|} H_1(\Gamma\smallsetminus I,\bQ)\right)(-\delta^{\aff}) \tensor  \bigwedge^{i-\delta^{\aff}}H^1(\widetilde{C}_{\times},\bQ) & \text{if }\delta^{\aff} \leq i \leq 2g(C_\eta)-\delta^{\aff}
\end{array}$$}
\end{lemma}
\begin{remark} \jc{For $2g(C_\eta)-\delta^{\aff}<i\leq 2g(C_\eta)$ the \jcc{highest weight quotients} are most easily described through the relative hard Lefschetz theorem \jcc{and \Cref{thm:hd_cj}}, as Tate twist of the $i^\prime$-th exterior power  for $i^\prime=2g(C_\eta)-i<\delta^{\aff}$. In this case the highest weight is $2g(C_\eta)$.}
\end{remark}

Lemma \ref{lem:bond_matroid} justifies the following definition:
\begin{definition}\label{def:GraphComplex}
Given a connected graph $\Gamma$, with set of edges $\rE$,
	we write $\mathscr{C}(\Gamma)$ for the collection of subsets of $\rE$ whose removal does not disconnect $\Gamma$.
	In other words, a subset $I\subseteq \rE$ belongs to $\mathscr{C}(\Gamma)$ if and only if $\Gamma\smallsetminus I$ is connected. 	
\end{definition} 
\begin{remark}\label{rem:matr}
In the literature {(e.g. \cite{bjorner}) the collection} $\mathscr{C}(\Gamma)$ is {called} the family of independent subsets in what is known as the {\em bond, or cographic, matroid } of the graph $\Gamma$. 
The set $\mathscr{C}(\Gamma)$ is partially ordered with respect to inclusions and we denote by $|\mathscr{C}(\Gamma)|$ the associated simplicial complex, i.e.,
the complex whose $k$-dimensional faces are the $(k+1)$-tuples of edges belonging to $\mathscr{C}(\Gamma)$. 
\end{remark}

The rank of $\mathscr{C}(\Gamma)$ is the cardinality of the complement of a spanning tree, namely
$|\rE|-|\rV|+1$, which, in the case of the dual graph of a nodal curve $C_\times$ equals 
$ \delta^{\aff}(C_\times)$, hence the simplicial complex {$|\mathscr{C}(\Gamma)|$ is of dimension $\delta^{\aff}(C_\times)-1$}. \jc{We denote \jcc{the} chain complex computing the reduced cohomology of $\mathscr{C}(\Gamma)$ by $\widetilde{\cC^\bullet(\Gamma)}$. }
%

Let us summarize:
\begin{proposition}\label{prop:highest_weight_cohomology}
Let $C_\times$ be a nodal stable curve, {$\Gamma$ its dual graph and $\delta^{\aff}:=\delta^{\aff}(C_\times)$. Let $B$ be an {\'etale} neighborhood  of $[C_\times] \in \overline{\ms{M}}_{g}$ on which there exists a universal family $\pi\colon\cC \to {B}$ and small enough such that the preimage $D$ of the boundary divisor $\overline{\ms{M}}_{g} \smallsetminus  \ms{M}_{g}$ in ${B}$ has irreducible components indexed by the nodes of $C_\times$. Then the fiber of $\IC(\bigwedge^i R^1)$ at the point $[C_\times]$ is given by the CKS complex ${\mathbf C}^{\bullet}( \{ N_j \}, \bigwedge^i H^1(\cC_\eta\jc{,\bQ}))$ (\cref{cks_general}) and }for every $i=0,{\dots,}2g\jc{(\cC_\eta)}$ we have
\jchange{
	\begin{enumerate}
		\item $H^{r}({\mathbf C}^{\bullet}( \{ N_j \}, \bigwedge^i H^1(\cC_\eta\jc{,\bQ})))=0 \text{ for } r > \min\{ i, \delta^{\aff}\}.$
		\item\label{item:highestweight}
		For $\delta^{\aff} \leq i \leq 2g(C_\eta)-\delta^{\aff}$ the highest weight quotient of ${\mathbf C}^{\bullet}( \{ N_j \}, \bigwedge^i H^1(\cC_\eta\jc{,\bQ}))$ is given by
		$$	Gr_{i+\delta^{\aff}}^W({\mathbf C}^{\bullet}( \{ N_j \}, \bigwedge^i H^1(\cC_\eta\jc{,\bQ})))= 
			\bigwedge^{i-\delta^{\aff}}H^1(\widetilde{C}_\times\jc{,\bQ})\otimes \widetilde{\cC^\bullet(\Gamma)}(-\delta^{\aff})$$
	\end{enumerate}}
\end{proposition}

\begin{proof}
The first statement follows from \cref{lem:bond_matroid}\jc{, as $\im N_{I}=0$ for $|I|>\delta^{\aff}$ and $N_I=0$ for $|I|>i$}.

\jc{To prove (\ref{item:highestweight}) we use the description of the highest weight quotient given in \cref{lem:bond_matroid}. As the differentials in the complex were defined in terms of the operators $N_e$ the highest weight quotient $Gr_{i+\delta^{\aff}}^W({\mathbf C}^{\bullet}( \{ N_j \}, \bigwedge^i H^1(\cC_\eta,\bQ)))$ for $i>\delta^{\aff}$ is the tensor product of the corresponding quotient for $i=\delta^{\aff}$ with $\bigwedge^{i-\delta^{\aff}}H^1(\widetilde{C}_\times,\bQ)$, and it therefore suffices to consider the case $i=\delta^{\aff}$.}

\jc{In this case observe that $\bigwedge^{\delta^{\aff}} H_1(\Gamma,\bQ)$ is one-dimensional,  since $\delta^{\aff}=\dim H_1(\Gamma,\bQ)$. Similarly, for any $I\in \cC(\Gamma)$ the space $\bigwedge^{\delta^{\aff}-|I|} (H_1(\Gamma-I,\bQ))$ is also one dimensional. So for every $k$, the degree $k$ part of the complex has a basis consisting of the non-disconnecting cardinality $k$ subsets of the edges set, namely precisely the $(k-1)$-cells of $|\mathscr{C}(\Gamma)| $. It is easy to check that the boundary maps coincide with the maps of the complex $\widetilde{\cC^\bullet(\Gamma)}$ that computes the reduced cohomology of $|\mathscr{C}(\Gamma)|$, which proves (\ref{item:highestweight}).}
\end{proof}
\begin{remark}\label{rem:purity}
\jc{As the local system $\wedge^i R^1$ is a pure local system, the intersection cohomology complex $\IC(\wedge^i R^1)$ is a pure complex and therefore the cohomology sheaves in degree $k$ are of weight $\leq i+k$. Thus  the cohomology of the highest weight quotient described in (\ref{item:highestweight}) above is concentrated in the top degree. We will see below (\cref{shellable_mat}, \cref{repsym}) that results on matroids allow to compute the cohomology of $|\cC(\Gamma)|$ directly and in particular to deduce that the cohomology in the top degree is non-zero.} 
\end{remark}

\subsection{Description of the summands supported on the partition strata}

{Before we apply the results let us recall a well known,} elementary estimate:
\begin{lemma}\label{lem:estimate}
Let $\pi: \cC \to B$ be a flat projective family of locally planar reduced curves of arithmetic genus $g$ such that  
compactified Jacobian family $\pi^J: \ov{J}_{\cC} \to B$, 
relative to a choice of a fine polarization exists and has nonsingular total space. Let ${U} \subset B$ be a dense open set such that  the restriction
$\pi\colon \cC_{{U}} \to {U}$ is smooth, and denote by $R^1$ the local system on ${U}$:
$$
R^1:={R^1\pi_* \QQ}_{|{U}}.
$$
Then
\begin{equation}\label{equ: estimate}
\cH^r(\IC(\bigwedge^i R^1 ))=0 \text{ for } r>i.
\end{equation}

\end{lemma}
\begin{proof}
It follows from the decomposition theorem for $\pi^J: \ov{J}_{\cC} \to B$ that {for $i=0,\dots 2\jc{g(\cC)}$ the complex }$\IC(\bigwedge^i R^1 ))[-i]$ is a direct summand in $\bR \pi^J\QQ$:
\begin{equation}\label{eq:dtpi}
\IC(\bigwedge^i R^1 ))[-i] \subset \bR \pi^J_{{*}}\QQ
\end{equation}
Assume $\cH^r(\IC(\bigwedge^i R^1 ))_b\neq 0$ for $b\in B$ and some $r >i$. It then follows from the relative Hard Lefschetz theorem that we may assume $i\geq g\jc{(\cC)}$.
Taking stalks of the cohomology sheaf $\cH^{r+i}$ at $b$ in the above \cref{eq:dtpi} we have
\begin{equation*}
0 \neq \cH^{r}(\IC(\bigwedge^i R^1 ))_{{b}}\subset H^{r+i}(\ov{J}_{\cC}(b)\jc{,\bQ}).
\end{equation*} 
which is a contradiction since $r+i>2i \geq 2g\jc{(\cC)}=2\dim \ov{J}_{\cC}(b)$.
\end{proof}

\jc{We can now prove one of the main results of this paper. To state the result let us recall that for a partition $\un{n}=(n_1,\dots,n_r)$ of $n$ we introduced the stratum $S_{\ven}^\times \subset \cA_n$ parametrizing spectral curves with $r$ smooth irreducible components $C_i$ of degree $n_i$ over $C$ intersecting transversally in $n_in_j(2g-2)$ nodes (\cref{lem:deltasn}) and the dual graph of these curves is denoted by $\Gamma_{\ven}$ (\cref{not:Gamma-n}).}
\begin{theorem}\label{thm:main1}
Let $\hn:\mn \to \an$ be the Hitchin map.
For every partition $\ven$ of $n$, the stratum $\sn$ is a support for all the sheaves
\begin{equation}\label{eq:range_supp}
^p\!\!{\mathscr H}^{{\jc{i}}}(\bR  {\hn}_{*}\QQ) 
\hbox{ with } \delta^{\aff}(\ven) \leq \jc{i} \leq {2\dim \an -\delta^{\aff}(\ven)} 
\end{equation}
More precisely, for every $\jc{i}$ in the range of (\ref{eq:range_supp}), there is a direct summand in 
$^p\!\!{\mathscr H}^{\jc{i}}(\bR  {\hn}_{*}\QQ) $ which is the intermediate extension of the 
local system $\cL_{\jc{i}, \ven}$  on the open set $ {\snode} \subset \sn$, 
whose stalk at a point $a\in \snode$ is 
\begin{equation*}\label{eq:loc_sys}
({\cL_{\jc{i}, \ven}})_a  = H^{\delta^{\aff}(\ven)-1}\left( |\mathscr{C}(\Gamma_{\ven})|,\bQ \right)(-\delta^{\aff}{(\ven)})\otimes \bigwedge^{\jc{i}-\delta^{\aff}{(\ven)}}H^1(\widetilde{C}_a\jc{,\bQ})
\end{equation*}
and underlying a variation of pure Hodge structures of weight $\jc{i}+\delta^{\aff}(\ven)$.
\end{theorem}

\begin{remark}
Theorem \ref{thm:main1} holds in the context of M. Saito's mixed Hodge modules (see \cite[Appendix]{drs}).
In particular, the resulting direct summands of  the  pure Hodge structures given by the cohomology groups $H^k(\mn,\QQ)$ (these are pure since they coincide with the cohomology of the nilpotent cone (fiber of the Hitchin map over the origin)\jchange{, see e.g. \cite[Theorem 1]{Heinloth} for a short argument}) are pure Hodge substructures. 
\end{remark}
\jc{\begin{remark}\label{rem:inductivedescription}
		The local systems ${\cL_{\jc{i}, \ven}}$ that determine the summands supported on $\sn$ arise as a tensor product. The monodromy of the combinatorial part $H^{\delta^{\aff}(\ven)-1}\left( |\mathscr{C}(\Gamma_{\ven})|,\QQ \right)$ turns out to be finite and we will determine it in \cref{cor:rank_loc_syst}. The cohomology of $\widetilde{C}_a$ is the sum of the cohomology groups of the irreducible components which are the generic spectral curves for the Hitchin fibration for bundles of rank $n_1,\dots,n_r$, so that the $\IC$ complexes of these local systems appear in the Hitchin fibration for the Levi subgroup $\prod \GL_{n_j} \subset \GL_n$. 
\end{remark}}

\begin{proof}
\jc{We have seen in \cref{prop:SuppOnlySn} that} either 
$^p\!\!{\mathscr H}^{\jc{i}}(\bR  {\hn}_{*}\QQ)$ has a summand which is {\em fully} supported at  $\sn$ or none of its summands intersect $\sn^{\jc{\times}}$, therefore
it is enough to consider a general point  $a\in \sn^{\jc{\times}}$ corresponding to a nodal spectral curve $C_a$ with $\jc{r}$ smooth components. Let $\Sigma_a$ as in  \cref{cor:map_to_moduli} be a transversal slice to $\sn$ at $a$. Since $\sn$ has codimension $\delta^{\aff}(\ven)$, we have $\dim \Sigma_a=  \delta^{\aff}(\ven)$. Furthermore,
by transversality, $\hn^{-1}(\Sigma_a)$ is nonsingular, and we have the diagram (\ref{equ:cartesian}). Let \jc{$\cU^{sm}\subset \cU$} be {the }open set where the universal curve 
$\pi: \cC_{\cU} \to \cU$ is smooth, and denote as before by $R^1$ the local system 
$$
R^1:={R^1\pi_* \QQ}_{\jc{\cU^{sm}}}.
$$
Since the family $\pi: \cC_{\jc{\cU}} \to \jc{\cU}$ is versal, we have, by \cite[Theorem 5.11]{MSV}, which we recalled in  \cref{thm:hd_cj}, that 
\begin{equation}\label{equ:dec_thm_vsl}
\bR \pi^J_* \QQ \simeq \bigoplus_{i=0}^{2\dn} \IC(\bigwedge^iR^1)[-i]. 
\end{equation}
By proper base change and the isomorphism in Diagram (\ref{equ:cartesian}) we have that
\begin{equation}
\mc{({\bR {\hn}_* \QQ})_{|\Sigma_a}   \simeq  \bR (\hn|_{\Sigma_a})_* \QQ} \simeq {f_{\Sigma_a}^*}(\bR \pi^J_* \QQ )
\end{equation}
is split semisimple. \jc{From \cref{cor:compwithversal} we also know that}
\begin{equation}\label{equ:perv_coh}
^p\!\!{\mathscr H}^{\jc{i}}({\bR {\hn}_* \QQ}_{|\Sigma_a}) = ^p\!\!{\mathscr H}^{\jc{i}}({f_{\Sigma_a}^*}\bR \pi^J_* \QQ) = {f_{\Sigma_a}^*} \IC(\bigwedge^iR^1).
\end{equation}
By stratification theory it is clear that
$\sn$ is a support for $^p\!\!{\mathscr H}^{\jc{i}}({\bR {\hn}_* \QQ})$ if and only if $a$ is a support for $^p\!\!{\mathscr H}^{\jc{i}}({\bR {\hn}_* \QQ}_{|\Sigma_a})$.
Since $\dim \Sigma_a=  \delta^{\aff}(\ven)$, by \cref{cor:point} and \cref{equ:perv_coh}, this happens if and only if
\begin{equation}\label{equ:new_sprt}
\cH^{\delta^{\aff}(\ven)}(\IC(\bigwedge^iR^1))_a \neq 0.
\end{equation}
By \cref{lem:estimate} \jc{ and the relative hard Lefschetz isomorphism} this is possible only if $\jc{\delta^{\aff}(\ven)\leq i\leq 2\dim \an - \delta^{\aff}(\ven)}$. On the other hand, 
when this is the case, \jc{\cref{prop:highest_weight_cohomology} \jcc{and \cref{rem:purity} }tell us that 
$$\cH^{\delta^{\aff}(\ven)}(\IC(\bigwedge^iR^1))_a \cong \bigwedge^{i-\delta^{\aff}}H^1(\widetilde{C}_a,\bQ)\otimes H^{\delta^{\aff}-1}(|\cC(\Gamma_{\ven})|,\bQ)(-\delta^{\aff}).$$  
This gives the claimed formula for the local system $\cL_{i,\ven}$ and to conclude we only need to observe that these are non-trivial. We already know that $\dim H^1(\widetilde{C}_a,\bQ)=2(g(C_a)-\delta^{\aff}(C_a))=2 \dim \cA_n - 2\delta^{\aff}(\ven)$ (\cref{lem:deltasn}), so that the contribution of this group is non-trivial for $i$ in the given range. The non-vanishing of the cohomology of the combinatorial complex $\cC(\Gamma_{\ven})$ will be recalled in \cref{repsym}.}
\end{proof}
\begin{corollary}\label{cor:fin_mon}
For every $\ven$, and for ${k}=\delta^{\aff}(\ven)$ \jc{or} ${k}={2\dim(\an)- \delta^{\aff}(\ven)}$, the pull back of the local systems 
${\cL_{{k}, \ven}}$ to \jc{the preimage $\cA^{\nod}_{\un{n}}$ of $\sn^\times$ in $\cA_{\ven}$} has trivial monodromy . 
\end{corollary}	
\begin{proof}
The irreducible components of the Jacobian of the spectral curve $C_a$ are indexed by the degrees of the restriction of the line bundles to the components of $C_a$. 
Therefore, the sheaf of irreducible components of $C_a$ is constant on  $\cA^{\nod}_{\un{n}}$. The local system 
\[
({\cL_{{2\dim(\an)- \delta^{\aff}(\ven)}, \ven}})_a  = H^{\delta^{\aff}(\ven)-1}\left( |\mathscr{C}(\Gamma_{\ven})|\jc{,\bQ} \right)(-\delta^{\aff}{(\ven)})
\otimes \bigwedge^{\jc{2 \dim \cA_n - 2\delta^{\aff}(\ven)}}H^1(\widetilde{C_a}\jc{,\bQ})
\]
\jc{appears in the top cohomology of $\bR {\hn}_* \QQ$ and} \jc{it is} a subsheaf of the $\QQ$-linearizaton of the sheaf of irreducible components \jc{(see  \cref{lem:SuppOnlySn})}, therefore its pullback to  $\cA^{\nod}_{\un{n}}$
is constant. This is also true for ${\cL_{\delta^{\aff}(\ven), \ven}}$ which is isomorphic to it (up to a Tate twist). 
\end{proof}
\begin{remark}\label{rem:fin_mon}
		The generic Galois group of the finite map 
	$\mult_{\un{n}}: \cA^{\nod}_{\un{n}} \longrightarrow S^{\nod}_{\un{n}}$
	is the subgroup of the symmetric group $S_r$ stabilizing the partition $\ven$ of $n$. Writing $\ven=1^{\alpha_1}\cdots n^{\alpha_n}$, i.e. letting $\alpha_i$ be the number of elements in $\ven$ equal to $i$,
	this subgroup is 
	\begin{equation}\label{subgroup}
	\prod_i S_{\alpha_i} \subseteq S_r.
	\end{equation}
	In particular the sheaves $\cL_{\delta^{\aff}( \ven), \ven}$ and  
		$\cL_{{2\dim\an - \delta^{\aff}(\ven)}, \ven}$ are constant if $n_i\neq n_j$ for all $i\neq j$.
\end{remark}
{We are now left to compute the rank of the local systems $\cL_{{k}, \ven}$ and determine their monodromy in order to show that \jc{some summands} do contribute to the cohomology of $\mn$.}

\subsection{The monodromy and the rank of the new local systems}\label{sec:mon_loc_sys}

{We start by the computation of the cohomology of the complex $|\cC(\Gamma_{\ven})|$ appearing in \cref{thm:main1}.
Recall from \cref{not:Gamma-n} that $\Gamma_{\ven}$ is a graph with multiple edges between any two vertices. 
As we remarked in \cref{rem:matr} the poset $\cC(\Gamma_{\ven})$ is by definition the collection of independent subsets of the bond matroid of the graph $\Gamma_{\ven}$ (for terminology about matroids see \cite{white}, but we will try to spell out the notions we use in the case we need).}  

For any matroid $M$ the simplicial complex $|\In(M)|$ of its independent subsets has special properties:
\begin{theorem}\cite[Theorem 7.3.3, Theorem 7.8.1]{bjorner}\label{shellable_mat}
The simplicial complex $|\In(M)|$ of independent subsets associated with a rank ${\delta}$ matroid 
 has the homotopy type of a bouquet of 
${\delta}-1$-dimensional  spheres.
\end{theorem}

%


%




Let us denote by $\Gamma_r$ the complete graph on $r$ vertices for $r\geq 2$. As before we will denote by $|\mathscr{C}(\Gamma_r)|$ the simplicial complex defined by the cographic matroid of $\Gamma_r$, i.e., its $k$-simplices are the subsets of $k+1$ edges of $\Gamma_r$ that do not disconnect the graph.    

The following result is a combination of well known results on matroids and a result of Stanley.
\begin{proposition}\label{repsym}
	For any $r \geq 2$ the cohomology group 
$H^{\mathrm{ top}}(|\mathscr{C}(\Gamma_{r})|\jc{,\bQ})$, has rank $(r-1)!$, and, with its natural structure of $S_r$-module, is isomorphic to the  representation induced by a primitive character of a maximal cyclic subgroup.
\end{proposition}
To deduce this result let us introduce the simplical complex $\Nspan(\Gamma_r)$ \jcc{ of  non-spanning subsets of the graphic matroid of  $\Gamma_r$, namely the subsets of edges not containing a spanning tree}. Let us denote by  $\Flat(\Gamma_r)$ the lattice of partitions of $\{1,\dots,r\}$ which in the language of matroids correspond to the poset of flats of the cographic matroid, because a flat in this case is a partition into complete subgraphs.  To this lattice one attaches the simplicial complex $\Delta(\Flat(\Gamma_r))$ whose $k-$simplices are chains of partitions $p_{disc} < p_1 < \dots < p_k < p_{triv}$ where $p_{disc}$ is the discrete partition and $p_{triv}$ is the trivial partition.   

We need the following result which is a general fact on matroids: 
\begin{lemma}\label{lem:cographicflat}
Let $N={r \choose 2}$ denote the number of edges of $\Gamma_r$ We have natural isomorphisms
\begin{equation}
H^i(|\mathscr{C}(\Gamma_{r})|\jc{,\bQ}) \simeq H_{N-3-i}(|\Nspan(\Gamma_r)|\jc{,\bQ}) \simeq H_{N-3-i}(|\Delta(\Flat(\Gamma_r))|\jc{,\bQ}).
\end{equation}
The second isomorphism is $S_r$-equivariant, while in the first isomorphism the $S_r$-rep\-re\-sen\-ta\-tions differ by the sign character. 
\end{lemma}
\begin{proof}
Consider the boundary of the complex of all subsets of the edges of $\Gamma_r$. Its geometric realization is the boundary of an $N-1$-simplex, i.e.\ an $N-2$-dimensional sphere.

The first isomorphism is the content of {\cite[Exercise 7.43 on page 278]{bjorner}} and amounts to combinatorial Alexander duality, once one notices that $\mathscr{C}(\Gamma_{r})$ and $\Nspan(\Gamma_r)$ are Alexander dual complexes in $\partial \Delta^{N-1}$ (see \cite{bjotan} for a quick proof of Alexander duality which is adapted to this context).  We see that the isomorphism is twisted by the sign representation considering the action of $S_r$ on the top cohomology of the ambient sphere (\cite[Theorem 2.4]{Stanley}).

The second isomorphism, due to Folkman, is (\cite[Theorem 3.1]{folk}), using that the set of edges of $\Gamma_r$ forms a crosscut of the partition lattice. To see that this isomorphism is $S_r$-equivariant we briefly recall Folkman's argument. 

Note that for any edge $e$ the subcomplex $L_e$ of $\Delta(\Flat(\Gamma_r))$ formed by the simplices that are contained in a simplex that satisfies $p_1=e$ is contractible. Moreover, for any non-spanning subset $I$ of edges the intersection $\cap_{e\in I} L_e$ is contractible to the $0$-simplex given by the partition defined by the subgraph $I$ (see \cite[Section 3]{folk}).    

Thus the cohomology of $\Delta(\Flat(\Gamma_r))$ can be computed from the nerve of the covering given by the subcomplexes $L_e$ and this agrees with the cohomology of $|\Nspan(\Gamma_r)|$. 
\end{proof}
\begin{proof}(of Proposition \ref{repsym})
Applying the previous lemma, the computation reduces to the computation of the homology of the lattice of partitions which was determined in \cite[Theorem 7.3]{Stanley} to be the  representation induced by a primitive character of a maximal cyclic subgroup tensored with the sign representation. ({See \cite[Section 6]{MoPe} for a more detailed exposition of the argument.})
\end{proof}

The dual graph $\Gamma_{{\ven}}$ of a spectral curve in $S_{\ven}^\times$ contains a complete graph on the vertices, but it will have multiple edges between the vertices.

Let us therefore fix some notation. Given a graph $\Gamma$ and $I$ a subset of edges let us denote by $\widehat{\Gamma}_{I}$ the graph obtained by doubling the edges in $I$, i.e. for every edge $e\in I$ we add an edge $\widehat{e}$ connecting the same vertices as $e$. 
 
\begin{proposition}
Let $\Gamma $ be a graph, let $I$ be a non-empty subset of edges. Let $  |\mathscr{C}(\Gamma)|$ and $|\mathscr{C}(\widehat{\Gamma}_{I})|$ be the simplicial complexes associated to $\Gamma$ and $\widehat{\Gamma}_{I}$. 
Then, for every $\ell$, there is a canonical isomorphism
\begin{equation}
\jc{H^{\ell}}\left( |\mathscr{C}(\Gamma)|,\bQ \right)\simeq \jc{H^{\ell+|I|}} ( |\mathscr{C}(\widehat{\Gamma}_{I})|,\bQ ).
\end{equation}
If a finite group $G$ acts on $\Gamma$ preserving $I$, \jc{the action extends to $\widehat{\Gamma}_I$ and} the isomorphism is $G$--equivariant \jc{with respect to the induced actions}.
\end{proposition}

\begin{proof}
It is a direct application of the deletion-contraction sequence: Let us first assume that $I=\{e\}$ consists of a single edge.
Then the set of faces in $\mathscr{C}(\widehat{\Gamma}_{I})$ is the disjoint union of the set of those which contain a doubled edge $\widehat{e}$ and those who don't. The subcomplex of those faces not containing $\widehat{e}$ is the simplicial complex of the graph $\widehat{\Gamma}_{I}/{\widehat{e}}$ obtained by removing $\widehat{e}$ and collapsing the vertices joined by $\widehat{e}$. We therefore get an exact sequence of \jc{co}chain complexes (\jcc{we again denote the reduced cochain complex of ${\mathscr{C}(\widehat{\Gamma})}$ by $\widetilde{\cC^\bullet(\Gamma)}$)}
\jc{\begin{equation}
	0 \longrightarrow \widetilde{\cC^{\bullet-1}(\Gamma)}\longrightarrow  \widetilde{\cC^\bullet(\widehat{\Gamma}_{I})} \longrightarrow  \widetilde{\cC^\bullet(\widehat{\Gamma}_{I}/\widehat{e})}  \longrightarrow 0.
	\end{equation}
}
%
Note that the edge $e$ becomes a loop in the graph $\widehat{\Gamma}_{I}/{\widehat{e}}$, hence $|\mathscr{C}(\widehat{\Gamma}_{I}/\widehat{e})|$ is a cone and has vanishing \jc{reduced co}homology.

By induction this shows that the \jc{$G$--equivariant} morphism \jc{$\widetilde{\cC^{\bullet-|I|}(\Gamma)}\to \widetilde{\cC^\bullet(\widehat{\Gamma}_{I})}$} induced by mapping those faces in $\widehat{\Gamma}_{I}$ that contain all of the doubled edges to its intersection with $\Gamma$ induces an isomorphism in \jc{co}homology.  
\end{proof}

\jc{We will apply this to the graph $\Gamma_{\ven}$ which can be obtained from the complete graph on $r$ vertices, by successively doubling subsets of edges that are preserved by the subgroup of $S_r$ that preserves the partition $\ven$. Thus the representation of this subgroup on $H^{\mathrm{ top}}(|\mathscr{C}(\Gamma_{\ven})|,\bQ)$ is the restriction of the representation of $S_r$ on the corresponding group for the complete graph described in \cref{repsym}. Thus we find:}

\begin{corollary}\label{cor:rank_loc_syst}
Let $\ven= n_1 \geq n_2 \geq \cdots \geq n_r=1^{\alpha_1}\cdots n^{\alpha_n}$ be a partition of $n$. The rank of the local system $\cL_{{\delta^{\aff}(\ven)+i}, \ven}$ is 
\begin{equation}
\mathrm{rank}\,\cL_{{\delta^{\aff}(\ven)+i}, \ven}=
(r-1)!\binom{{2(\dim\an-\delta^{\aff}{(\ven)})}}{i}.
\end{equation}
The monodromy of the (isomorphic) local systems ${\cL_{\delta^{\aff}(\ven), \ven}}$ and 
${\cL_{{2\dim\an - \delta^{\aff}(\ven)}, \ven}}$ is given by the restriction to the subgroup 
$\prod_i S_{\alpha_i} \subseteq S_r$ 
of the representation of $S_r$ induced by a primitive character of a maximal cyclic subgroup. 
{In particular, if $\ven$ is a partition with pairwise distinct $n_i$ the monodromy of these sheaves is trivial, so that the corresponding summand of $\bR {\hn}_{*}\QQ$ contributes to the cohomology of $\mn$. }
\end{corollary}

\begin{remark}\label{rem:rank2}
If $n=2$ and $\Gamma_{\jc{(1,1)}}$ is the graph with two vertices joined by $2g-2$ edges, it is immediately seen that $ |\mathscr{C}({\Gamma }_{\jc{1,1}})|$ is a sphere of dimension $2g-4$. The corresponding representation is, for $g\neq 2$, the sign representation. \jc{Similarly f}or $g=2$ we have a zero-dimensional sphere, namely two points, and the relevant representation is th\jc{e sign representation} on {\em reduced} cohomology.

\jc{For $n=2$, $g\geq 2$ and the partition  $\un{n}=(1,1)$ the normalization $\widetilde{C_a}$ of a spectral curve $C_a$ in $S_{(1,1)}^\times$ is a disjoint union of two copies of $C$. Thus we have $H^1(\widetilde{C_a},\bQ)=H^1(C,\bQ)\oplus H^1(C,\bQ)$ and the monodromy of the corresponding system is the permutation representation induced from interchanging the components, i.e., the representation of $S_2$ is the sum of $2g$ trivial representations and $2g$ sign representations. Therefore for $0<j<4g$ the sign representation appears in $\bigwedge^jH^1(\widetilde{C_a},\bQ)$ and thus $\cL_{i,(1,1)}$ has non zero invariant sections for all $i$ satisfying $\delta^{\aff}+1\leq i \leq 2\dim \an -\delta^{\aff}-1 $. In particular these $\cL_{i,(1,1)}$ contribute non trivially to the cohomology of $\cM_2$.}
\end{remark}

\section{Appendix: The derivative of the Hitchin morphism is dual to the derivative of the action}\label{appe}

The duality statement from the title of the section is certainly known, but we could not find a reference for it. Although we only apply the result for the group $\GL_n$ it turns out that the proof is most easily explained in the more general setting of Higgs bundles for reductive groups. This is because in the case of $\GL_n$ it is easy to loose track of implicit identifications between the Lie algebra and its dual.

\subsection{Reminder on $G$-Higgs bundles}

We keep working over $\bC$ and use our fixed smooth projective curve $C$. In addition let $G$ be a connected reductive group with Lie algebra $\cg=\Lie(G)$. We will denote the dual of $\cg$ by $\cg^*$.

Given a $G$-torsor $\cP\to C$ and a representation {$\rho\colon G \to {\GL(V)}$} with $V$ a finite dimensional complex vector space, we will denote by $\cP(V):=\cP\times^G V$ the associated vector bundle.  

\jchange{Of course, if $G=\GL_n$, then the frame bundle $\cP= \Isom(\cO^n,\cE)$ of a vector bundle $\cE$ is a $\GL_n$-torsor and we get $\cE$ back by taking $V=\comp^n$ to be the standard representation. In this case $\cP(\cg)=\cEnd(\cE)\cong \cP(\cg^*)$.} 

A $G$-Higgs bundle on $C$ is a pair $(\cP,\phi)$ where $\cP\to C$ is $G$-torsor and $\phi \in H^0(C,\cP(\cg^*)\tensor K_C)$ is a global section of the coadjoint bundle twisted by $K_C$. We denote by 
$$ \Higgs_G := \left\langle (\cP,\phi) \,| \; \cP \in \Bun, \phi \in H^0(C,\cP(\cg^*)\tensor K_C) \right\rangle$$
the stack of $G$-Higgs bundles over $C$, which is the cotangent stack to the stack of $G$-bundles on $C$.

\begin{remark}
	The above definition follows the convention of \cite{BeilinsonDrinfeld}. In the literature on $G$-Higgs bundles it is also common to choose a $G$-invariant inner product $(,)$ on $\cg$ and use it to identify $\cg \cong \cg^*$. To state the results in an invariant form it seems to be most convenient to avoid this choice. As a consequence we will formulate some notions for the dual $\cg^*$ that are commonly used for $\cg$ for Higgs bundles, i.e., to use coadjoint orbits instead of adjoint orbits. 
\end{remark}

Let us recall from \cite{NgoHitchin} how to view $G$-Higgs bundles as sections of a morphism of stacks.

\begin{lemma}\label{Lem:NgoHiggsDescription} The category of Higgs bundles $(\cP,\phi)$ on $C$ is equivalent to the category of 2-commutative diagrams
$$\xymatrix{
& [\cg^*/G\times \bG_m]\ar[d]\\
C \ar[r]_-{K_C}\ar[ur]^-{(\cP,\phi)} & B\bG_m,
}$$
where $B\bG_m$ is the classifying stack of line bundles, $K_C$ is the map defined by the canonical bundle on $C$  and $[\cg^*/G\times \bG_m]$ is the quotient stack defined by the product of the coadjoint action of $G$ on $\cg^*$ and the standard scaling action of $\bG_m$ on the vector space $\cg^*$.
\end{lemma}
\begin{proof}
This is not hard to unravel: By definition a $G$-torsor on $C$ is the same as a map $C \to BG =[\Spec{\bC}/G]$, so the pair $\cP,K_C$ defines a map $C\to [B (G\times \bG_m)].$ Now for any representation $\rho \colon G\times \bG_m \to \GL(V)$ the associated bundle is the pull back of the morphism $[V/G\times\bG_m]$ and applying this to the representation on $\cg^*$ we see that $\cP(\cg^*)\tensor K_C = C\times_{B(G\times \bG_m)} [\cg^*/(G\times \bG_m)]$. Therefore the datum of a section of this bundle is equivalent to a section of 
$$\xymatrix{
	& [\cg^*/G\times \bG_m]\ar[d]\\
	C \ar[r]_-{(\cP,K_C)}\ar[ur]^-{(\cP,\phi)} & B(G\times \bG_m).
}$$
\end{proof}

\subsection{Deformations of $G$-Higgs bundles}
As the main aim of the section is to compare derivatives of morphisms from and to $\Higgs_G$ we need to recall the basic results on deformations of Higgs bundles.

To a Higgs bundle $(\cP,\phi)$ we attach the complex of vector bundles on $C$
	$$\cC(\cP,\phi):=[ \cP(\cg) \map{\ad^*(\underline{\;})(\phi)} \cP(\cg^*) \tensor K_C],$$
where $\ad^*\colon \cg \times \cg^* \to \cg^*$ denotes the coadjoint action of $\cg$ on $\cg^*$.

\begin{lemma}[\cite{Nitsure}]
The tangent space of the deformation functor of $G$-Higgs bundles at $(\cP,\phi)\in \Higgs_G$ is given by $H^1(C,\cC(\cP,\phi))$ and automorphisms of deformations that extend the identity of $(\cP,\phi)$ are given by $H^0(C,\cC(\cP,\phi))$.
\end{lemma}
\begin{proof}
	The deformation theory argument for the computation of the tangent space to $\Higgs_G$ can be found in \cite{Nitsure}. In the language of \cref{Lem:NgoHiggsDescription} we have a cartesian diagram:	
	$$\xymatrix{
		[\cg^* \tensor K_C/G]\ar[d]^{p_{K_C}}\ar[r] & [\cg^*/G\times \bG_m] \ar[d]^p\\
		C \ar[r]^-{K_C} & B\bG_m =[\Spec k/\bG_m]
	}$$ 
	and Higgs bundles are sections of the map $p_{K_C}$. 
	
	Now the tangent stack to any quotient stack $[X/G]$ can be described as the quotient of the complex of $G$-vector bundles $\Lie(G) \times X \to TX$ on $X$, which we think of a complex in degree $[-1,0]$. 
	
	Therefore the tangent complex to the stack $[\cg^*/G]$ (which lives in degree $[-1,0]$) is given by the $G$-equivariant complex 
	$$[\cg \map{\ad^*} \cg^*]$$
	on $\cg^*$ and thus the tangent complex to $p_{K_C}$ over $\cg^*\tensor K_C$ is given by $$[\cg\tensor \cO_C \to \cg^* \tensor K_C].$$ 
	
	Deformations of $(\cP,\phi)$ are deformations of the corresponding section $(\cP,\phi) \colon C\to [\cg^* \tensor K_C/G]$ and the pull back of the tangent complex at this section is $$[\cP(\cg)  \map{\ad^*(\underline{\;})(\phi)} \cP(\cg^*) \tensor K_C].$$ 
\end{proof}		

\begin{remark}
	For any Higgs bundle $(\cP,\phi)$ the complex $\cC(\cP,\phi)$ is self-dual with respect to the duality defined by $\cHom(\,\cdot\, , K_C[1])$.
	Therefore Serre-duality induces pairings 
	$$H^i(C,\cC(\cP,\phi)) \times H^{2-i}(C,\cC(\cP,\phi)) \to \bC$$
	that for $i=1$ define the standard 2-form $\omega_{\Higgs}$ on $\Higgs_G=T^*\Bun$.
\end{remark}

\subsection{The Hitchin morphism}

The Hitchin morphism for $G$-Higgs bundles is defined as follows. Denote by $\chi$ the quotient map $$\chi\colon \cg^* \to \cg^*/\!/G =\car^*,$$ where $\car^*=\Spec (\Sym^\bullet \cg)^G$. 

\begin{remark}
	As usual, a choice of homogeneous invariant polynomials would give an isomorphism $\car^* \cong \Spec k[f_1,\dots,f_r] \cong \bA^r$ identifying $\car^*$ with an affine space. The map $\chi$ is equivariant with respect to the $\bG_m$ action on $\cg^*$ and the induced action on $\Spec (\Sym^\bullet \cg)^G$, whose weights are given by the degrees of the invariant polynomials $f_i$.
\end{remark}

We denote by $\car^*_{K_C}=(\cg^*\times K_C/\!/G) \to C$ the corresponding affine bundle and by 
$$\cA_G := H^0 (C, \car^*_{K_C})$$
the base of the Hitchin morphism. Again, any choice of invariant polynomials for $G$ defines an isomorphism $\cA_G \cong \oplus_i H^0(C,K_C^i)$, but it will be more convenient to avoid such a choice.

The map $\overline{\chi}\colon[\cg^*/G\times \bG_m] \to [\car^*/\bG_m]$ then induces a map
$$h_G\colon \Higgs_G \to \cA_{{G}}= H^0(C,\car^*_{K_C}),$$
which is often denoted as $h_G(\cP,\phi)=:\chi(\phi)$.

\subsection{The regular centralizer (local version)}

To define the Hitchin morphism and the analog of the action of the Jacobian of the spectral curve we now recall the construction of the regular centralizer groups from \cite{NgoHitchin}. 

Let us fix the standard notations. The group $G$ acts on $\cg$ via the adjoint action, which we will denote by $\Ad\colon G \to \GL(\cg)$, the derivative of this action is denoted $\ad\colon \cg \to \End(\cg)$. Similarly $\Ad^*\colon G \to \GL(\cg^*)$ denotes the dual action given by $\Ad^*(g)(\phi)(\underline{\quad}):=\phi(\Ad(g)^{-1}.\underline{\quad})$, so that its derivative is ${\ad}^*(X)=-\ad(X)^t$.  

For an element $\varphi\in \cg^*$ we denote its centralizer in $G$ by 
$C(\varphi):=\{ g\in G | \Ad^*(G)(\varphi)=\varphi\}$ and by $\cg^{\varphi}:=\{A\in \cg | \ad^*(A)(\varphi)=0\}$ its Lie algebra. The groups $C_G(\varphi)$ define a group scheme 
$$C_{\cg^*} :=\{(g,\varphi)\in G\times \cg^*| \Ad^*(g)(\varphi)=\varphi\} \to \cg^*$$ over $\cg^*$. The set of regular elements $\cg^{*,\reg}\subset \cg^*$ is defined to be the subset of those elements for which $\dim C_G(\varphi)=\rank(G)$ is minimal.

The restriction $C_{\cg{*,reg}}$ of $C_{\cg^*}$ to the space of regular elements descends to a group scheme $J_{\car^*}$ on $\car^*=\cg^*/\!/ G$, called the regular centralizer. The group scheme $J_{\car^*}$ comes equipped with a natural map 
$$m\colon \chi^*J_{\car^*} \to C_{\cg^*} \subset G\times \cg^*$$ 
which is defined to be the unique regular map extending the natural isomorphism $\chi^*J_{\car^*}|_{\cg^{*,\reg}} \cong C_{\cg^{*,reg}}$. We denote by $dm$ the induced map on Lie algebras $$dm \colon \chi^*Lie(J_{\car^*}) \to \Lie(C_{\cg^*} )\to  \cg \times \cg^*.$$

\begin{notation}
	As in \cite{NgoHitchin} we will need to keep track of the action of the multiplicative group $\bG_m$ on our objects. We will denote by $\bC(n)$ the one dimensional vector space with the $\bG_m$ action given by the $n$-th power of the standard action. For any vector bundle $E$ with a $\bG_m$-action we will denote by $E(n):=E\tensor \bC(n)$.
\end{notation}

\begin{remark}\label{rem:dm}
	On $\cg^*$ the group $\bG_m$ acts by scalar multiplication which induces an action on $\car^*=\cg^*/\!/G$. The action on $\cg^*$ also preserves centralizers and thus induces an action on $C_{{\cg^*}}$, given by $$t.(g,\varphi):=(g,t\varphi).$$ In particular this action preserves $\cg^{*,\reg}$ and thus $C_{\cg^{*,\reg}}$ even descends to a group $\overline{J}$ over $[\car^*/\bG_m]$.
	
	Note that the formula for the $\bG_m$ action shows that the derivative
	\begin{equation*}
dm\colon \chi^*Lie(J_{\car^*}) \to  \cg \times \cg^*
	\end{equation*} is equivariant for the $\bG_m$--action that on $\cg \times \cg^*$ is given by the trivial action on the first factor $\cg$ and the standard action on the second factor $\cg^*$. Therefore, identifying $\cg(-1) \times \cg^* \cong T^*\cg^*$ we can interpret $dm$ as a morphism
	\begin{equation}\label{eq:dm}
	dm\colon \chi^*Lie(J_{\car^*})(-1) \to  T^*\cg^* 
	\end{equation}
	The restriction of this map to $\cg^{*,\reg}$ is injective, as $m$ was injective over $\cg^{*,\reg}$.
\end{remark}

\begin{remark}\label{rem:dchi}	
	The map $\chi\colon \cg^{*} \to \car^*$ is by definition $G$-invariant and equivariant with respect to the $\bG_m$ action, therefore its derivative
	\begin{equation}\label{eq:dchi}
	d\chi \colon \cg^* \times \cg^* = T\cg^* \to \chi^* T\car^*
	\end{equation}
	is also equivariant with respect to the induced $\bG_m$ action and the restriction 
	$$d\chi|_{\cg^{*,\reg}} \colon T \cg^{*,\reg} = \cg^* \times \cg^{*,\reg} \to \chi^* T\car^*|_{\cg^{*,\reg}}$$ is surjective, because the map $\chi\colon \cg^{*} \to \car^*$ admits a section $\kappa\colon \car^* \to \cg^{*,\reg}\subset \cg^*$ called the Kostant section.		
\end{remark}

The following observation is the group theoretic origin of the duality result for the Hitchin fibration.
\begin{lemma}\label{lem:LocalPairing}
	The canonical pairing $$\langle \,, \, \rangle \colon T\cg^* \times_{\cg^*} T^*\cg^* \to \bC$$
	induces a $G\times\bG_m$-equivariant perfect pairing 
    $$\chi^*Lie(J_{\car^*})|_{\cg^{*,\reg}}(-1) \times_{\cg^*}  \chi^* T\car^*|_{\cg^{*,\reg}} \to \bC(0)$$
    and thereby an isomorphism
   $$\Lie(J)^*(1) \cong T\car^*.$$    
\end{lemma}
\begin{proof}
	From {Remarks \ref{rem:dm} and \ref{rem:dchi}} we know that $\chi^*Lie(J_{\car^*})|_{\cg^{*,\reg}}(-1)$ is a subbundle of $T^*\cg^{*,\reg}$ and $\chi^* T\car^*|_{\cg^{*,\reg}}$ is a quotient of $T\cg^{*,\reg}$ and both have the same dimension.
	
	As the map $\chi$ is constant on $G$-orbits, the tangent space to a $G$ orbit is in the kernel of $d\chi$, i.e., for every $\varphi\in \cg^*$
	$$ V_\varphi:=\im( \cg \map{\ad^*(\underline{\;})(\varphi)} T_\varphi\cg^* =\cg^* )\subset \ker(d\chi).$$
	If $\varphi\in\cg^{*,\reg}$ is regular we have $\dim V_\varphi = \dim \cg/\cg^\varphi = \dim \cg - \dim \car^*$. As $d\chi$ is surjective in this case we find $V_\varphi=\ker(d\chi)$ for $\varphi\in \cg^{*,\reg}$. 
	
	Now $G$-invariance of the pairing $\langle \,, \, \rangle$ i.e., $\langle g.\varphi,g.A\rangle=\langle\varphi,A\rangle$ for all $g\in G,\varphi\in \cg^*,A\in\cg$ implies that for all $X\in\cg$ we have
	$$\langle \ad^*(X)(\varphi),A\rangle=\langle\varphi,-\ad(X)(A)\rangle=-\langle\ad^*(A)(\varphi),X\rangle.$$ 
	This implies that $V_\varphi^\perp=\cg^\varphi$ and this implies our claim. 
\end{proof}

\begin{remark}\label{LocalComputationForGLn}
	For $G=\GL_n$ the above can be rephrased in terms of coordinates. In this case $\car^*\cong \bA^n$ is the space of characteristic polynomials of matrices. In order to compute the differential $d\chi$ of the map $\chi \colon \gl_n \to \car^*$ it is convenient to choose the coordinates $\chi(\varphi):=(\frac{1}{i}\Trace(\varphi^i))_{i=1\dots n}$. Then $d\chi_\varphi\colon \gl_n \to k^n$ is given by $X \mapsto (\Trace(\varphi^{i-1}X))_{i=1 \dots n}$.
	
	The regular centralizer group scheme can also be described explicitly: For any monic polynomial $p(t)\in k[t]$ we define $J_p:=(k[t]/p(t))^*$ as the unit group of the algebra $k[t]/p(t)$, which defines an $n$-dimensional commutative group scheme $J$ over $\bA^n$. As a matrix $\varphi$ is regular if and only if its characteristic polynomial $p_\varphi$ is its minimal polynomial, we see that the assignment $J_{p_\varphi} \to \GL_n$ given by $f(t) \mapsto f(\varphi)$ is injective for regular matrices $\varphi$ and therefore identifies $J_{p_\varphi}$ with the centralizer of $\varphi$. By definition of the regular centralizer the map $\chi^*(J)\to I \subset \GL_n \times \gl_n$ is given by the unique extension of the canonical map on $\gl^{\reg}$. As the formula $f(t)\mapsto f(\varphi)$ is well defined for all $\varphi$ it gives this extension.
	
	We also observe that $s \in \bG_m$ acts on $\car^*$ by $p\mapsto s.p$, where $s.p$ is the polynomial given by multiplying the coefficient of $t^{n-i}$ by $s^i$. This lifts to an action $J_p \to J_{s.p}$, given by $t\mapsto st$ and this is compatible with the above map $f(t) \mapsto f(\varphi)$. 
	
	Note that $\Lie(J_p)\cong k[t]/(p(t))$ (as $1+\epsilon f(t)$ is an invertible element of $k[\epsilon,t]/(\epsilon^2,p(t))$ for all $f$). Finally the standard basis $1,t,\dots t^{n-1}$ of $k[t]/p(t)$ defines an isomorphism $\Lie(J) \cong k^n \times \car$.
	
	Thus for any $\varphi$ the map $dm \colon k^n \cong \Lie(J_p) \to \gl_n$ is given by $(a_i)\mapsto \sum_{i=0}^{n-1} a_i \varphi^i$. 
	
	Finally, we use the pairing $(A,B):=\Trace(AB)$ on $\gl_n$. With respect to this form the dual of the map $k^n \cong \Lie(J_p) \to \gl_n$ is therefore given by $$X \mapsto (\Trace(\varphi^iX))_{i=0\dots n-1}$$ which is $d\chi_\varphi$.	
\end{remark}

\begin{remark}\label{rem:LocalVersionOfDuality}
	We can reformulate the above Lemma as a duality statement on $[\cg^*/G]$: As for any quotient stack, the tangent stack to this quotient is defined by the complex 
	$$ [\cg \times \cg^* \map{(A,\phi) \mapsto (\ad^*(A)(\phi),\phi)} T\cg^* = \cg^* \times \cg^*],$$
	i.e., the quotient stack of these bundles is the pull-back of the tangent stack to $\cg^*$. This complex is self-dual up to a shift by $1$.
	
	Considering $\chi$ as a morphism $\overline{\chi}\colon [\cg^*/G] \to \car^*$ the differential becomes the morphism 
	$$  [\cg \times \cg^* \map{(A,\phi) \mapsto (\ad^*(A)(\phi),\phi)} \cg^* \times \cg^*] \map{(0,d\chi)} [ 0 \to T\car^*].$$
	
	Similarly as the morphism $\chi^*J \to I$ is $G$-equivariant (because $J$ was defined by descending $I|_{\cg^{*,\reg}}$), the morphism $dm$ defines a $G$-equivariant morphism $$[\chi^*\Lie(J) \to 0] \map{(dm,0)} [\cg \times \cg^* \map{(A,\phi) \mapsto (\ad^*(A)(\phi),\phi)} \cg^* \times \cg^*].$$
	
	 \cref{lem:LocalPairing} says, that these morphisms are $\bG_m$-equivariantly dual to each up to a shift by $1$ of the complex and twisting the action by $(1)$.
\end{remark}
\subsection{The regular centralizer (global version)}

Let us recall the global version of the regular centralizer as explained in \cite[Section 4]{NgoHitchin}: We saw that the regular centralizer defines a group scheme $\overline{J}$ on $[\car^*/\bG_m]$, that we can pull back to a group scheme $J_{\car^*_{K_C}}$ on $\car^*_{K_C}=C \times_{B\bG_m} [\car^*/\bG_m]$ which we pull back via the tautological map $\cA_G\times C\to \car^*_{K_C}$ to define a group scheme $J_{\cA_G}$ on $\cA_G\times C$. 

Similarly, the pull back of the sheaf of centralizers $\overline{C}_{\cg^*}$ on $[\cg^*/G\times \bG_m]$ under the classifying map $\Higgs_G \times C \to [\cg^*/G\times \bG_m]$ is denoted $C_{\Higgs\times C}$. By construction $C_{\Higgs\times C}=\Aut(\cE_{\univ},\phi_{\univ})$ is identified with the group of $G$-automorphisms of the universal Higgs bundle $(\cE_{\univ},\phi_{\univ})$ that preserve $\phi_{\univ}$.

The map $\chi^*J \to C_{\cg^*}$ therefore induces a natural morphism $$\iota\colon(h\times id_C)^*J_{\cA} \to \Aut(\cE_{\univ},\phi_{\univ}).$$ 

Over $\cA_G$ one defines the group scheme $P_{\cA_G}$ of $J_{\cA_G}$-torsors on $C$, i.e., at a point $a\colon C \to \car^*_{\Omega}$ is given by the torsors of the group scheme of $a^*J_{\cA_G}$ on $C$.  Then $\iota$ induces an action $\act\colon P_{\cA_G}\times_{\cA_G} \Higgs_G \to \Higgs_G$.

\subsection{The duality statement}
We can now formulate the main result of this section:
\begin{proposition}\label{prop:dactdhdual}
	There exists a canonical isomorphism $\Lie(P_{\cA_G}/\cA_G) \cong T^*\cA_G$ such that the morphisms 
	$$d\act\colon h^*\Lie (P_{\cA_G}/\cA_G) \to T\Higgs_G$$
	and $$dh\colon T\Higgs_G \to h^*(T\cA_G)$$
	become dual to each other with respect to the symplectic form $\omega_{\Higgs}$.
\end{proposition}

\begin{proof}
	The result follows from the local statement  \cref{lem:LocalPairing} as follows:
	The regular centralizers $P_a$ are defined to be $a^*J_{\cA}$-torsors, so $\Lie(P_a)=H^1(C,\Lie(a^*J_{\cA_G}))$. 
	
	For any Higgs bundle $(\cE,\phi)$ the action of $\act_{(\cE,\phi)}\colon \Lie(P_a) \to h^{-1}(a)$ is induced from $\iota\colon(h\times id_C)^*J_{\cA_G} \to C_{(\cE,\phi)}\subset \Aut(\cE/C)$. 
	
	Therefore applying $\Lie$ we find that the differential of the action is induced from the morphism of complexes 
	$$ [\Lie(a^*J_{\cA}) \to 0 ] \map{(d\act,0)} [ \ad(\cE) \map{\ad^*()(\phi)} \ad(\cE)^* \tensor \Omega]$$
	after passing to $H^1$.
	
	
	By  \cref{lem:LocalPairing} we know that this map of complexes is up to tensoring with $K_C[-1]$ this map is dual to the map
	$$[ \ad(\cE) \map{\ad^*(\underline{\;})(\phi)} \ad(\cE)^* \tensor K_C] \map{(0,d\chi)} [0 \to T\car^*_{K_C}]$$
	that induces $dh$ by  \cref{rem:LocalVersionOfDuality}. Therefore applying Serre-duality to $H^1$ of the above complexes we obtain the proposition. 
\end{proof}

\end{document}